\documentclass[amsart,11pt]{article}
\usepackage{fancyhdr}
\usepackage[top=1in,bottom=1in,left=1in,right=1in]{geometry}
\usepackage[all]{xy}
\usepackage{amssymb,amsthm,amsmath,amscd}
\usepackage{epsfig}
\usepackage{subfigure}
\newtheorem{theorem}{Theorem}[section]

\newtheorem{corollary}{Corollary}[section]

\newtheorem{lemma}{Lemma}[section]
\newtheorem{proposition}{Proposition}[section]
\newtheorem*{theorem*}{Theorem}
\newtheorem*{remark*}{Remark}
\newtheorem{remark}{Remark}

\def\reals{{\mathbb{R}}}
\def\pds{{\mathrm{p}_d \mathrm{sin}}}

\def\WW{{\mathcal{U}_{\lambda}}}

\def\BB{{B}}

\def\LL{{L}}

\newcommand{\be}{\begin{eqnarray}}
\newcommand{\ee}{\end{eqnarray}}
\newcommand{\ben}{\begin{eqnarray*}}
\newcommand{\een}{\end{eqnarray*}}
\def\MM{{\mathrm{M}}}

\DeclareMathOperator{\dist}{dist}
\DeclareMathOperator{\diam}{diam}

\DeclareMathOperator{\supp}{supp}
\DeclareMathOperator{\pdms}{{\mathrm{p}_{{\it d}-1}\mathrm{sin}}}

\def\Supp{{\supp(\mu)}}

\newcommand{\cmplx}{\mathbb C}

\newcommand{\nats}{\mathbb N}

\def\RR{{\mathbb{R}}}
\def\HH{{H}}
\newcommand{\di}{\, \mathrm{d}}
\newcommand{\fdi}{\mathrm{d}}
\def\HH{{H}}
\def\Span{{\mathrm{Span}}}

\def\ud{{\mathrm{d}}}
\def\RR{{\mathbb{R}}}

\def\Tbe{{\mathrm{T_{ube}}}}

\def\MIn{{\mathrm{min}}}

\def\Supp{{\mathrm{supp}(\mu)}}

\def\DIst{{\mathrm{dist}}}

\begin{document}
\title{High-Dimensional Menger-Type Curvatures - Part II:\\
$d$-Separation and a Menagerie of Curvatures}

\author{Gilad Lerman\thanks{The authors have been supported by
NSF grant \#0612608}~\thanks{Contact:\{lerman,whit0933\}@umn.edu} \
and J. Tyler Whitehouse$\hspace{0.01cm}^{*} \hspace{-0.03cm}^{\dag}$
}

\maketitle

\begin{abstract}
We estimate $d$-dimensional least squares approximations of an
arbitrary $d$-regular measure $\mu$ via discrete curvatures of $d+2$
variables. The main result bounds the least squares error of
approximating $\mu$ (or its restrictions to balls) with a $d$-plane
by an average of the discrete Menger-type curvature over a
restricted set of simplices. Its proof is constructive and even
suggests an algorithm for an approximate least squares $d$-plane. A
consequent result bounds a multiscale error term (used for
quantifying the approximation of $\mu$ with a sufficiently regular
surface) by an integral of the discrete Menger-type curvature over
all simplices. The preceding paper (part I) provided the opposite
inequalities of these two results. This paper also demonstrates the
use of a few other discrete curvatures which are different than the
Menger-type curvature. Furthermore, it shows that a curvature
suggested by L{\'e}ger (Annals of Math, 149(3), p.~831-869, 1999)
does not fit within our framework.
\end{abstract}
%
%\newline

\noindent AMS Subject Classification (2000):
%28A75,
60D05, 49Q15, 42C99\\ \\
\noindent Keywords: least squares $d$-planes, multiscale geometry,
Ahlfors regular measure, uniform rectifiability, polar sine,
Menger-type curvature.
%, 46T99
%
%\vspace{1cm}
%%%%%%%%%%%%%%%%%%%%%%%%%%%%

\section{Introduction}
We propose an approximate construction for the least squares
$d$-plane of a $d$-regular measure (defined below) and an  estimate
for the corresponding least squares error via Menger-type curvatures
as well as other curvatures.

There are two different kinds of motivation for this investigation.
The first one concerns clustering of data points sampled around
intersecting $d$-planes~\cite{spectral_theory, spectral_applied}
(and even more general $d$-dimensional manifolds~\cite{ACL-kscc}).
Here a possible approach is to assign local discrete curvatures (of
at least $d+2$ points) that distinguish the different clusters
(i.e., they are ``sufficiently small'' within each cluster and
``large'' for points of mixed clusters). Indeed, some of the results
of this paper and the preceding one show that within each cluster
the discrete curvatures discussed here (or its
variants~\cite{spectral_applied, LW-volume}) are tightly controlled
on average by the $d$-dimensional least squares error of the
cluster.

Another use of the current study is in the constructive
approximation of sufficiently regular $d$-dimensional surfaces for
$d$-regular measures~\cite{Jo90,DS91,DS93}, and we will expand on it
later.

\subsection*{Basic Setting}
Our current setting includes a real separable Hilbert space $H$ of
possibly infinite dimension, an intrinsic dimension $d \in \nats$,
where $d < \dim (H)$, and a $d$-regular measure $\mu$ on $H$ (or
equivalently, a $d$-dimensional Ahlfors regular measure). That is, a
locally finite Borel measure $\mu$ with a constant $C \geq 1$ such
that for all $x\in\Supp$ and $0 < t \leq\diam(\Supp)$ with the
corresponding closed ball $B(x,t)$:
\begin{equation}\label{equation:measure-comp}C^{-1}\cdot
t^d\leq\mu(B(x,t))\leq C\cdot t^d.\end{equation}
The smallest constant $C$ satisfying
equation~\eqref{equation:measure-comp} is denoted by $C_{\mu}$ and
is called the {\em regularity constant} of $\mu$.

This abstract setting allows estimates that are independent of the
ambient dimension. In Section~\ref{section:conclusion} we discuss
possible extensions of some of the estimates below to other
settings.
%, while allowing their dependence on the ambient dimension.
%Nevertheless, the techniques presented here for the toy model of
%$d$-regularity are sufficiently.
% general so that it is not hard to
%extend and further develop them (see e.g., \cite{LW-volume}).

\subsection*{Simplices and their Menger-type Curvature}

%In order to define and manipulate discrete curvatures of $d+2$
%variables,
We often work with $(d+1)$-simplices in $H$. We represent such a
simplex with vertices $x_0, \ldots, x_{d+1}$ by $ X = (x_0,\ldots
,x_{d+1}) \in H^{d+2}$. We usually do not distinguish between the
$(d+1)$-simplex and its representation $X$.
The following functions of $X$ will be frequently used. The largest
and smallest edge lengths of $X$ are denoted by $\diam (X)$ and
$\min(X)$ respectively. The $d$-content $\MM_{d+1}(X)$ is the
multiplication of the $(d+1)$-volume of the simplex $X$ by $(d+1)!$.
Equivalently, it is the $(d+1)$-volume of any parallelotope
generated by the vertices of $X$. The polar sine of $X = (x_0,\ldots
,x_{d+1})$ at the vertex $x_i$, $0\leq i\leq d+1$, is the function
$$
\pds_{x_i}(X) = \frac{\MM_{d+1}(X)}{\displaystyle \prod_{\substack{0\leq j\leq d+1\\
j\not=i}}\|x_j-x_i\|}\, ,
$$
where it is zero if the denominator is zero.  When $d=1$, the polar
sine (hereafter abbreviated to p-sine) reduces to the absolute value
of the ordinary sine of the angle between two vectors.

The primary discrete curvature we have worked with  is
$$
c_{\mathrm{MT}}(X) =
\sqrt{\frac{1}{d+2}\cdot\frac{1}{\diam(X)^{d(d+1)}}\sum^{d+1}_{i=0}
\pds^2_{x_i}(X)}\,,$$
where it is zero if the denominator is zero. Other possible
curvatures are introduced in Section~\ref{section:curvatures} as
well as in~\cite{LW-volume}.

\subsection*{Least Squares $d$-Planes via Curvatures}

We fix a location $x \in \supp(\mu)$ and a scale $0 < t \leq
\diam(\supp(\mu))$ with the corresponding ball $B=B(x,t)$ in $H$ (by
a ball we always mean the closed ball). The scaled least squares
error of approximating $\mu$ by $d$-planes at $B$ has the form
$$
\beta_2(B) =\sqrt{\inf_{d \text{-planes } L}
\int_B\left(\frac{\dist(y,L)} {\diam(B)}\right)^2 \frac{\di \mu
(y)}{\mu (B)}} \,,
$$
where it is zero if $\mu(B)=0$. If $B=B(x,t)$, then we often denote
$\beta_2(B)$ by $\beta_2(x,t)$.
%\end{document}
Any $d$-plane $L$ that minimizes the above infimum is referred to as
a least squares $d$-plane for the restriction of $\mu$ onto $B$,
i.e., $\mu|_B$. Our main result concerns the estimation of
$\beta_2(B)$ by Menger-type curvatures. Before stating it we develop
the specific integrals of curvatures to be used and review a
previous result.

We arbitrarily fix $0<\lambda<2$ and define the following set of
simplices in $B = B(x,t)$ with sufficiently large edge lengths
\begin{equation}\label{equation:big-U}U_{\lambda}(B)=\left\{X\in
B^{d+2}:\min(X)\geq\lambda\cdot t\right\}.\end{equation}
We then define the following squared curvature of $\mu$ at
$B=B(x,t)$ as a function of $x$ and $t$:
\begin{equation}
%\label{equation:little-def}
\nonumber
c_{\mathrm{MT}}^2(x,t,\lambda)=\int_{U_{\lambda}(B(x,t))}c_{\mathrm{MT}}^2(X)\di\mu^{d+2}(X)\,.\end{equation}

In~\cite{LW-part1} we controlled from below the local least squares
approximation error for $\mu$ as follows.
%estimated
%$\beta^2_2(B)$ from above by $c_{\mathrm{MT}}(\mu|_B,\lambda)$ in the following
%way.
%
\begin{proposition}\label{proposition:upper-bound-big-scale}There exists a constant
$C_0=C_0(d,C_{\mu})\geq1$ such that
\begin{equation*}\label{equation:upper-big-scale}c_{\mathrm{MT}}^2(x,t,\lambda) \leq
\frac{C_0}{\lambda^{d(d+1)+4}}\cdot\beta_2^2(x,t)\cdot\mu(B(x,t)),\end{equation*}
for all $\lambda>0$,  $x\in\Supp,$ and $0<t\leq\diam(\Supp)$.
\end{proposition}

Here we establish the following opposite inequality. Its
constructive proof also suggests an algorithm for an approximate
least squares $d$-plane.
\begin{theorem}\label{theorem:doozy}There
exist constants $0<\lambda_0=\lambda_0(d,C_{\mu})<2$ and
$C_1=C_1(d,C_{\mu})\geq1$ such that
$$\beta_2^2(x,t) \cdot\mu(B(x,t)) \leq C_1\cdot c_{\mathrm{MT}}^2(x,t,\lambda_0) \,,$$
for all $x\in\Supp$ and $0<t\leq\diam(\Supp)$.
\end{theorem}
%

%Our interest in Proposition~\ref{proposition:upper-bound-big-scale}
%and Theorem~\ref{theorem:doozy} goes beyond the theory of uniform
%rectifiability.

In~\cite{LW-volume} we extend both
Proposition~\ref{proposition:upper-bound-big-scale} and
Theorem~\ref{theorem:doozy} to more general measures (e.g.,
truncated Gaussian distributions around Lipschitz graphs). We
actually use there different scalings of both the discrete
Menger-type curvature and the underlying integral, which allow us to
replace $U_\lambda(B)$ by $B^{d+2}$.

%%%%%%%%%% the construction

%\subsection*{Theoretical Algorithm for Least Squares $d$-Plane}
%
%\section{Algorithm}Step 1:  Find a $d$-separated collection of balls
%in $B(x,t)$, denoted by $\{B_0,\ldots,B_{d+1}\},$ and select the
%first $(d+1)$ balls, $B_0,\ldots, B_d$.
%
%Step 2:  Take $\widetilde{X}(d+1)\in\prod_{i=0}^d\frac{1}{2}\cdot
%B_i$, and for arbitrary $\rho_1>0$ and $\rho_2>0$ form the sets
%$$\mathcal{E}(\rho_1)=\left\{\widetilde{X}(d+1):\sum_{y\in U_{\lambda_0}(x,t|\widetilde{X}(d+1))}
%\frac{\pds_{\widetilde{x}_0}^2(\widetilde{X}(y,d+1))}{\diam(\widetilde{X}(y,d+1))^{d(d+1)}}\leq\frac{\rho_1}{t^{d(d+1)}}\cdot
%c_{\mathrm{MT}}^2(x,t,\lambda_0)\right\},$$and\begin{multline*}\mathcal{A}(\rho_2)=\Bigg\{\widetilde{X}(d+1):\max_{0\leq
%i\leq d}\sum_{(y,z)\in U_{\lambda_0}(x,t|\widetilde{X}(i;d+1))}
%\frac{\pds_{\widetilde{x}_0}^2(\widetilde{X}(y,i;z,d+1))}{\diam(\widetilde{X}(y,i;z,d+1))^{d(d+1)}}\leq\\\frac{\rho_1}{t^{d^2}}\cdot
%c_{\mathrm{MT}}^2(x,t,\lambda_0)\Bigg\}.\end{multline*}
%
%Step 3:  Find $\rho_1$ and $\rho_2$ such that
%$$\mathcal{E}(\rho_1)\cap\mathcal{A}(\rho_2)\not=\emptyset,$$and
%take any $\widetilde{X}(d+1)$ in this intersection, and form
%$L[\widetilde{X}(d+1)]$.  (optional: find the minimal such
%$\widetilde{X}(d+1)$.)
%
%(optional) Step 4: Minimize steps 1-3 over the $(d+2)$ different
%choices of balls, i.e., instead of taking $B_0,\ldots, B_d$, take
%all such choices and find the minimum.
%

\subsection*{Underlying Regular Surfaces via Curvatures}

The problem of approximating $\mu$ by a $d$-plane and estimating the
corresponding least squares error extends to approximating $\mu$ by
a ``sufficiently regular $d$-dimensional surface'' and estimating
the ``cumulative error'' of multiscale least squares approximations.
We demonstrate those notions for an arbitrarily fixed ball $B
\subseteq H$.

The cumulative error of multiscale least squares approximation, or
equivalently the Jones-type flatness~\cite{Jo90,DS91,DS93,LW-part1},
is defined for the restriction of $\mu$ to $B$ as follows:
\begin{equation*}
%\label{equation:jone-flat-int}
\nonumber J_{2}(\mu|_B)= \int_B\int^{\diam(B)}_0
\beta^2_2(x,t)\frac{\fdi t}{t}\di \mu (x)\,.\end{equation*}
The technical notion of a sufficiently smooth $d$-dimensional
surface, or more precisely an $\omega$-regular surface for an $A_1$
weight $\omega$, is presented in~\cite{DS91, DS93} (see also
\cite[Section~6]{LW-part1}).

David and Semmes~\cite{DS91, DS93} showed that there exists such a
surface containing the restriction of the support of $\mu$ to $B$ if
and only if there is a finite uniform bound on the quantities $\{
J_{2}(\mu|_{B'})/\mu(B')\}_{B' \subseteq B}$. If any of the two
equivalent conditions hold, then $\mu|_B$ is called uniformly
rectifiable.

Uniform rectifiability thus asks for a nice parametrization of
$\mu|_B$ (by a sufficiently regular $d$-dimensional surface). In
fact, the supporting theory~\cite{DS91, DS93} suggests a multiscale
construction of such a parametrization by least squares
approximations at different balls in $B$ (in the spirit
of~\cite{Jo90}). Furthermore the smoothness of the underlying
surface, i.e., the sizes of its parameters, can be controlled by the
least uniform bound of $\{ J_{2}(\mu|_{B'})/\mu(B')\}_{B' \subseteq
B}$. This is analogous to characterizing the smoothness of functions
by quantities based on wavelet coefficients or similar
Littlewood-Paley estimates as demonstrated in~\cite{D91_book}
and~\cite[Subsection~1.3]{DS93}.

While uniform rectifiability is tailored for the restricted setting
of $d$-regular measures, it is a rich ground for various notions of
quantitative geometry, which can be further extended and applied in
more general and practical settings. For example, insights of curve
construction of uniform rectifiability have been used
in~\cite{Lerman03,Lerman07Bioinformatics, LMM08} for rather
practical measures. The idea was to avoid approximation of the whole
support of the given measure by a curve (as done in uniform
rectifiability), but parametrize only a large fraction of the
support~\cite{Lerman03}, i.e., allowing an outlier component, and
moreover constructing a strip around the main
curve~\cite{Lerman07Bioinformatics, LMM08} in order to ``cover'' a
noisy component around it.

Our current work extends uniform rectifiability by using discrete
curvatures instead of least squares approximations of different
scales and locations. For this purpose we use the following squared
Menger-type curvature of $\mu|_B$:
$$c_{\mathrm{MT}}^2 \left(\mu|_B\right)
=\int_{B^{d+2}}c_{\mathrm{MT}}^2(X)\di\mu^{d+2}(X)\,.$$
We also extend this definition as follows.  We arbitrarily fix
$0<\lambda\leq 1$ and define the following set of well scaled
simplices in $B$:
\begin{equation}\label{equation:big-W}W_{\lambda}(B)=\{X\in
B^{d+2}: \ \min(X)\geq\lambda\cdot\diam(X)>0\}.\end{equation}
The squared Menger-type curvature of $\mu|_B$ with respect to the
parameter
 $\lambda$ has the form
$$c_{\mathrm{MT}}^2\left(\mu|_B, \lambda\right)
=\int_{W_{\lambda}(B)}c_{\mathrm{MT}}^2(X)\di\mu^{d+2}(X)\,.$$
Clearly, for any $\lambda>0$ we have that
\begin{equation*}\label{equation:part-tot}c_{\mathrm{MT}}^2\left(\mu|_B, \lambda\right) \leq
c_{\mathrm{MT}}^2\left(\mu|_B\right).\end{equation*}

In~\cite{LW-part1} we established the following result.
\begin{theorem}\label{theorem:upper-main}
There exists a constant $C_2=C_2(d,C_{\mu})\geq1$ such that
$$c_{\mathrm{MT}}^2\left(\mu|_B\right)\leq C_2\cdot
J_2 \left(\mu|_{6\cdot B}\right)$$
for all balls $B\subseteq H$.\end{theorem}

Using the constant $\lambda_0$ of Theorem~\ref{theorem:doozy} we
establish here an opposite inequality to
Theorem~\ref{theorem:upper-main} as follows.
\begin{theorem}\label{theorem:funky-salmon}There
exists a constant $C_3=C_3(d,C_{\mu})\geq1$ such that
\begin{equation*}\label{equation:funky-salmon}
J_2 \left(\mu|_B\right) \leq C_3 \cdot
c_{\mathrm{MT}}^2\left(\mu|_{3B}, \lambda_0 / 2 \right)
\end{equation*}for all balls $B\subseteq
H$ with $\diam(B)\leq\diam(\Supp)$.\end{theorem}
We thus obtain that the Jones-type flatness of $\mu|_B$ can be
replaced by the squared Menger-type curvature of $\mu|_B$. That is,
multiscale least squares approximations for uniform rectifiability,
in particular for surface reconstruction, could be replaced by using
Menger-type curvatures of various kinds of simplices.

%%%%%
% subsection on other curvatures...
\subsection*{A Curvature of L{\'e}ger and Related Methodology}

%In Section~\ref{section:curvatures} we introduce other curvatures
%satisfying estimates analogous to those given above.
%In addition,
%we also mention a curvature for which we can establish partial
%estimates which are still sufficient to quantify uniform
%rectifiability.

A previous curvature for studying the rectifiability and uniform
rectifiability of measures for $d\geq1$ was proposed by
L{\'e}ger~\cite{Leger}. However, his analysis was done only for
$d=1$ and when we try to generalize it to $d>1$ we find that his
curvature controls a quantity which characterizes a property weaker
than uniform rectifiability (see Subsections~\ref{subsection:p-q}
and~\ref{subsection:leger}).

Despite the fact that L{\'e}ger's curvature is unsuitable for our
purposes, his basic analysis is instrumental in this paper. Similar
to his work in~\cite{Leger}, the main ingredients of our analysis
include repeated applications of both Fubini's Theorem and
Chebychev's inequality as well as various metric inequalities and
identities. However, we needed to develop additional analytic and
combinatoric propositions for the case where $d>1$. In particular,
we have generalized the separation of points by pairwise distances
employed in~\cite{Leger} to a $d$-dimensional separation of
simplices (see Section~\ref{section:geom-prop}).  This
$d$-separation plays a fundamental role in the proof of
Theorem~\ref{theorem:doozy}, which is the main result of this paper.
Weaker notions of $d$-separation have been applied earlier by David
and Semmes~\cite[Lemma~5.8]{DS91} and
Tolsa~\cite[Lemma~8.2]{tolsa_riesz_rect}.

\subsection*{Organization of the Paper}
This paper is organized as follows. Section~\ref{section:context}
provides the preliminary notation and some related elementary
propositions. Section~\ref{section:geom-prop} states and proves a
geometric proposition regarding the $d$-dimensional separation of
points in the support of an arbitrary $d$-regular measure $\mu$ on
$H$. Sections~\ref{section:doozy-1} and~\ref{section:fubani} contain
the proofs of Theorem~\ref{theorem:doozy} and
Theorem~\ref{theorem:funky-salmon} respectively.
Section~\ref{section:curvatures} discusses other possible curvatures
and their relation to the curvature $c_{\mathrm{MT}}$, as well as
some problems with the curvature suggested by
L{\'e}ger~\cite{Leger}. Finally, Section~\ref{section:conclusion}
suggests new and mostly open directions to extend this work.

\section{Basic Notation and  Definitions}
\label{section:context}
\subsection{Main Context and Notational Conventions} \label{subsec:main_context}
We fix a real separable Hilbert space $\HH$, and denote its inner
product,  induced norm, and the dimension (possibly infinite) by
$<\cdot,\cdot>$, $\|\cdot\|,$ and $\dim(H)$ respectively. For $m \in
\nats$, we denote the Cartesian product of $m$ copies of $H$ by
$H^m$.
%We also
%refer to $H^n$, for some $n \in \nats$, $n \geq 3$, which is the
%Cartesian product of $n$ copies of $H$.

If $A\subseteq H$, then we denote its diameter by $\diam(A)$. If
$\mu$ is a measure on $H$, then its support is denoted by
$\supp(\mu)$. For $\mu$ and $A\subseteq H$, we denote the
restriction of $\mu$ to $A$ by $\mu|_A$.

We denote the closed ball in $\HH$, centered at $x\in\HH$ and of
radius $t$, by $B(x,t)$. If both the center and radius are
indeterminate, then we use the notation $B$.

We summarize some notational conventions  as follows. We typically
denote
%scalars by lower-case letters, e.g., $d$, $\sigma$,
scalars with values at least  $1$ by upper-case plain letters, e.g.,
$C$;
%Scalars With Small Real Values By $\Varepsilon$ And
%$\Delta$;
arbitrary integers  by lower case letter, e.g., $i,j$ and large
integers by $M$ and $N$; and arbitrary real numbers by lower-case
Greek or lower-case letters, e.g., $\rho$, $r$.

We reserve  $x$, $y$, and $z$ to denote elements of $H$; $X$ to
denote elements of $H^m$ for $m \geq 3$; $L$ for a complete affine
subspace of $H$ (possibly a linear subspace);
 $V$ to denote a complete linear subspace of $H$; $B$ to denote closed balls in
 $H$; and $t$ for arbitrary length
scales, in particular, radii of balls (we always assume that
$t\in\RR$, even when writing $0 < t \leq \diam(\supp(\mu))$ and when
$\supp(\mu)$ is unbounded).

We say that a real-valued function $f$ is controlled by a
real-valued function $g$, which we denote by $f \lessapprox g$, if
there exists a positive constant $C$ such that $f \leq C \cdot g$.
Similarly, $f$ is comparable to $g$, denoted by $f \approx g$, if $f
\lessapprox g$ and $g \lessapprox f$. The constants of control or
comparability may depend on some arguments of $f$ and $g$, which we
make sure to indicate if unclear from the context.

More specific notation and definitions commonly used throughout the
paper, as well as related propositions, are described in the
following subsections according to topic.

\subsection{Elements of $H^{n+1}$ and Corresponding
Notation}\label{subsection:element-notation}Fixing  $n \geq 1$, we
denote an element of $\HH^{n+1}$ by $X=(x_0,\ldots,x_{n})$. For $0
\leq i \leq n$, we let $(X)_i=x_i$ denote the projection of $X$ onto
its $i^{\text{th}}$ $H$-valued {\em coordinate}. The $0^{\text{th}}$
coordinate $(X)_0=x_0$ is special in many of our calculations.

For $0\leq i\leq n$ and $X=(x_0,\ldots,x_{n})\in\HH^{n+1}$, let
$X(i)$ be the following element of $H^n$:
\begin{equation}\label{equation:removal} X(i)=(x_0,\ldots,
x_{i-1},x_{i+1},\ldots,x_{n}),
\end{equation}
that is, $X(i)$ is the projection of $X$ onto $\HH^{n}$ that
eliminates its $i^{\text{th}}$ coordinate. Furthermore, for $n\geq2$
and $0\leq i<j\leq n$,  let $X(i;j)$ be the following element of
$H^{n-1}$:\begin{equation}\label{equation:double-removal}
X(i;j)=(x_0,\ldots,x_{i-1},x_{i+1},\ldots,x_{j-1},x_{j+1},\ldots,x_n).\end{equation}
If $1\leq i\leq n$, $X(i)\in H^n$, and $y\in H$, we form $X(y,i)\in
H^{n+1}$ as follows:
\begin{equation}
\nonumber X(y,i)=(x_0,\ldots,x_{i-1},y,x_{i+1},\ldots,x_n)\,.
\end{equation}Finally, if $y\in H$ and $z\in H$, $n\geq2$, and $1\leq i<j\leq n$, then we define the
elements $X(y,i;j),\, X(i;z,j)\in H^n$ and $X(y,i;z,j)\in H^{n+1}$
by the following formulas:
\begin{equation}\label{equation:double-remove-single-add}X(y,i;j)=(x_0,\ldots,x_{i-1},y,x_{i+1},
\ldots,x_{j-1},x_{j+1},\ldots,x_n),\end{equation}
\begin{equation}\label{equation:double-remove-single-add-last}X(i;z,j)=(x_0,\ldots,
x_{i-1},x_{i+1},\ldots,x_{j-1},z,x_{j+1},\ldots,x_n),\end{equation}
and\begin{equation}\label{equation:double-remove-double-add}X(y,i;z,j)=(x_0,\ldots,x_{i-1},y,x_{i+1},
\ldots,x_{j-1},z,x_{j+1},\ldots,x_n).\end{equation}

\begin{remark}\label{remark:notation-convention}We usually take $n=d+1,$ and
without referring directly to $X\in H^{d+2}$ we often denote
elements of $H^{d+1}$ by $X(i)$ for some $1\leq i\leq d+1$  and
elements of $H^d$ by $X(i;j)$ where $0\leq i<j\leq d$. \end{remark}
\subsection{Notation for Subsets of $\HH$}
\label{section:basic-definitions-1}For a ball $B(x,r)$ and
$\gamma>0$, let $\gamma\cdot B(x,r)=B(x,\gamma\cdot r)$. If $x\in H$
and $r_1, r_2
>0$, then we define the annulus
$$
A(x,r_1,r_2)=\left\{y\in H: r_1 < \|x-y\|\leq r_2 \right\} =
B(x,r_2) \setminus B(x,r_1).
$$

For  $n \geq 1$ and $X\in H^{n+1}$, we say that $X$ is {\em
non-degenerate} if the set $\left\{x_1-x_0,\ldots,
x_{n}-x_0\right\}$ is linearly independent, and we say that $X$  is
degenerate otherwise. For $X\in H^{d+1}$ let $L[X]$ denote the
affine subspace of $H$ of minimal dimension containing the vertices
of $X$, i.e., the coordinates of $X$, and let $V[X]$ be the linear
subspace parallel to $L[X]$.

If $V$ is a linear subspace of $H$, we denote its orthogonal
complement by $V^{\perp}$. If $L$ is a complete affine subspace of
$\HH$ and $x\in\HH$, we denote the distance between $x$ and $L$ by
$\dist(x,L)$.
%We denote the dimension of $L$ (as above) by $\dim(L)$.
%Similarly, if $L'$ is another affine subspace of $H$, we denote
%the distance between $L$ and $L'$ by
%$$\dist(L,L')=\inf_{x\in L,y\in L'}\left\{\|x-y\|\right\}.$$
If  $n\leq\dim(H)$, we use the phrase $n$-plane to refer to an
$n$-dimensional affine subspace of $H$.

If $\eta\geq0$ and $L$ is a complete affine subspace of $H$, we
define the tube of height $\eta$ on $L$ to be %%
\begin{equation*}
\label{equation:def-tube} \nonumber \Tbe(L,\eta)=\left\{y\in\HH:
\dist(y,L)\leq \eta\right\}.\end{equation*}

%\subsection{Content and Edge Length Functions }

%Given  $n \geq 2$ and  $X \in H^{n+1}$, its $n$-content  is given by
%
%\begin{equation}\label{equation:def-content-funct}
%
%\mathrm{M}_{n}(X)= \left(
%
%\mathrm{det}\left[
%
%\big\{\langle x_i-x_0,x_j-x_0\rangle\big\}_{i,j=1}^{n}
%
%\right]
%
%\right)^\frac{1}{2}\, .
%\end{equation}
%It is the $n$-dimensional Lebesgue measure of a parallelotope
%generated by the images of the vertices of $X$ under any isometric
%embedding of $L[X]$ in $\RR^n$.

\subsection{Elementary Properties of $d$-regular Measures
}\label{section:$d$-regular-prop}

We describe here basic properties of the $d$-regular measure $\mu$.
The following bound (extending the upper bound of $d$-regularity for
all radii) is straightforward:
\begin{equation}\label{equation:upper-regularity}
\mu(B(x,t)) \leq C_\mu \cdot t^d \ \text{ for all } \ x \in
\supp(\mu)\textup{ and }t>0.
\end{equation}

The following lemma is derived immediately from
equation~(\ref{equation:measure-comp}).
\begin{lemma}\label{lemma:$d$-regular-measure-2}
If $x\in\Supp$, $0<t \leq \diam(\supp(\mu))$, and
$0<s<1/C_{\mu}^{2/d}$, then
\begin{equation}
\label{equation:annulus-ball} \nonumber
%(1-\gamma^d\cdot C^2_\mu)\leq
%\frac{\mu(A(x,\gamma \cdot r_1, r_1))}{\mu(B(x,r_1))}\leq 1,
%(1-\gamma^d\cdot C^2_\mu)\leq
\frac{\mu(A(x,s \cdot t, t))}{\mu(B(x,t))}
%\leq 1,
\geq \left(1-s^d\cdot C^2_\mu\right)>0.
\end{equation}
\end{lemma}

The following proposition requires a little more work and is
established in~\cite{LW-semimetric}.
\begin{proposition}\label{proposition:artube}
If  $m \in \mathbb{N}$ is such that $1\leq m<d$, $\mu$ is a
$d$-regular measure on $\HH$ with regularity constant $C_{\mu}$,
 $0 \leq \epsilon \leq 1$, and $\LL$ an $m$-dimensional affine subspace of
$\HH$, then for all $x\in\mathrm{supp}(\mu)\cap\LL$ and $0<t\leq
\mathrm{diam}(\mathrm{supp}(\mu))$
\begin{equation*}
\label{eq:artube} \mu(\Tbe(\LL,\epsilon \cdot t )\cap \BB(x,t))\
\leq\ 2^{m+\frac{3\cdot d}{2}}\cdot C_{\mu}\cdot \epsilon^{d-m}
\cdot t^d.\end{equation*}

\end{proposition}

\subsection{Elementary Properties of the Polar Sine and the Menger-Type Curvature}
\label{section:polar-sine}

If $n\geq1$, $X\in H^{n+2}$, and $1\leq i \leq n+1$, then let
$\theta_i(X)$ denote the elevation angle of $x_i-x_0$ with respect
to $V\left[X(i)\right]$. We note that if $\min(X)>0,$ then
\begin{equation}\label{equation:elevation-sine}\sin(\theta_i(X)) = \frac{\dist(x_i,L[X(i)])}
{\|x_i-x_0\|}\,.
\end{equation}

The $d$-dimensional p-sine satisfies the following product
formula~\cite{LW-semimetric}:
\begin{proposition}\label{proposition:product-sine}If
$X=(x_0,\ldots,x_{d+1})\in\HH^{d+2}$ and $1\leq i\leq d+1$, then
\begin{equation*}\label{equation:product-sine}
\pds_{x_0}(X)=\sin\left(\theta_i(X)\right) \cdot
\pdms_{x_0}\left(X(i)\right).
\end{equation*}
\end{proposition}

Proposition~\ref{proposition:product-sine} and
equation~(\ref{equation:elevation-sine}) imply the following lower
bound for the p-sine.

\begin{lemma}\label{lemma:hey-diddle-diddle}If $x\in H$, $0<t<\infty$, and  $X\in B(x,t)^{d+2}$ is such
that $$\MM_d\left(X(i)\right)\geq \omega\cdot t^d\textup{  for some
} 0<\omega\leq 1 \textup{ and }1\leq i\leq d+1,$$
then$$\pds_{x_0}(X)\geq\frac{\omega}{2^{d+1}}\cdot\frac{\DIst(x_i,L[X(i)])}{t}.$$
\end{lemma}

The definition of the p-sine implies the following generalization of
the one-dimensional law of sines: If $X=(x_0,\ldots,x_{d+1})\in
H^{d+2}$ is such that $\min(X)>0,$ then
\begin{equation}\label{equation:law-of-sines}\frac{\pds_{x_i}(X)}{\displaystyle\prod_{\substack {0\leq s<r\leq d+1\\s,r\not=i}}\|x_s-x_r\|}
=\frac{\pds_{x_j}(X)}{\displaystyle\prod_{\substack{0\leq \ell<q\leq
d+1\\\ell, q\not=j}}\|x_{\ell}-x_q\|} \ \ \textup{  for all }0\leq
i<j\leq d+1.\end{equation}

Finally, we recall the following expression for
$c_{\mathrm{MT}}^2(\mu|_B)$ which was established
in~\cite{LW-part1}:
\begin{equation*}\label{equation:integral-simplification-1}
c_{\mathrm{MT}}^2\left(\mu|_B\right)=\int_{B^{d+2}}\frac{\pds_{x_0}^2(X)}
{\diam(X)^{d(d+1)}}\di\mu^{d+2}(X).
\end{equation*}

\subsection{Jones-Type Flatness for $1\leq p<\infty$}\label{subsection:betas}

For any fixed $1\leq p<\infty$, $x\in H$ and $0<t<\infty$, we define
the $d$-dimensional Jones'  numbers~\cite{DS91} as follows:
\begin{equation*}\label{equation:continuous-beta-1}\beta_p(x,t)=\begin{cases}
\displaystyle\inf_{d-\textup{planes }
L}\left(\int_{B(x,t)} \left(\frac{\dist(y,L)}{2\cdot
t}\right)^p\frac{\fdi\mu(y)}{\mu(B(x,t))}\right)^{1/p},
&\textup{ if }\mu(B(x,t))>0;\\
\quad\quad\quad\quad\quad\quad\quad\quad\quad\quad 0,&\textup{ if
}\mu(B(x,t))=0.\end{cases}\end{equation*}

For any fixed $1\leq p<\infty$ and any ball $B$ in $H,$ we define
the continuous local Jones-type flatness as follows:
\begin{equation*}\label{equation:jone-flat-int}
J_p(\mu|_B)= \int_B\int^{\diam(B)}_0 \beta^2_p(x,t)\frac{\fdi
t}{t}\di \mu (x).\end{equation*}

\section{On $d$-Separation of $d$-Regular Measures}\label{section:geom-prop}
We introduce here a notion of $d$-dimensional separation of
$(d+1)$-simplices and show that there are many such separated
simplices in $\Supp^{d+2}$. Specifically, we show that independently
of $x\in\Supp$ and $0<t\leq\diam(\Supp)$ there exists a
``sufficiently large'' amount of $(d+1)$-simplices, $X\in
[B(x,t)\cap\Supp]^{d+2}$, whose $d$-dimensional faces,
$\{X(i)\}_{i=0}^{d+1}$, are ``sufficiently large''. We refer to this
property as $d$-separation of the measure $\mu$ and also refer to
the corresponding simplices as $d$-separated. Similar notions were
already applied by David and Semmes~\cite[Lemma~5.8]{DS91} and
Tolsa~\cite[Lemma~8.2]{tolsa_riesz_rect}.

\subsection{$n$-Separated Simplices}\label{subsection:bus-time}Let $X\in H^{d+2}$ with $\diam(X)>0$. We say that $X$ is
{\em $1$-separated for}  $\omega>0$ if
$$\frac{\min(X)}{\diam(X)}\geq \omega\,.$$We say that $X$ is
{\em $d$-separated  for}  $\omega>0$ if
$$\frac{\min_{0\leq i\leq d+1}\ \MM_d(X(i))}{\diam^d(X)}\geq\omega\,.$$
More generally,  we say that  $X$ is {\em $n$-separated for}
$\omega>0$ and $1<n<d$ if the minimal $n$-content through its
vertices scaled by $\diam^n(X)$ is larger than $\omega$. We
typically do not mention the constant $\omega $, and we just say
$n$-separated if $\omega$ is clear from the context.

We note that the $n$-separation of an element $X$ implies the
$j$-separation for all $1\leq j< n$.  For example, given a
$d$-separated element $X$, using the product formula for contents we
have that
\begin{equation*}\label{equation:beasley}\omega\cdot\diam^d(X)\leq\min_{0\leq i\leq
d+1}\,\MM_d(X(i))\leq\min(X)\cdot\diam^{d-1}(X).\end{equation*}Hence,
$X$ is $1$-separated for  $\omega$.

%In the rest of this section we establish the existence of
%sufficiently many $d$-separated elements in $\Supp^{d+2}$.  In
%Subsection~\ref{subsection:separation-beginning} we extend the
%notion of $n$-separation of an element $X\in H^{d+2}$ to $d$-regular
%measures, and we make the claim that $d$-regular measures are
%$d$-separated.  We prove this claim in
%Subsection~\ref{subsection:sep-2}.
\subsection{$n$-Separated Balls and
Measures}\label{subsection:separation-beginning}

Let $B(x,t)\subseteq H$,  $m,n\in\nats$, $m\geq n\geq 1$, and
$\omega>0$. We say that a collection of $m+1$ balls,
$\{B_i\}_{i=0}^m$, is {\em $n$-separated in $B(x,t)$ for $\omega$}
if$$\bigcup_{0\leq i\leq m}B_i\subseteq B(x,t),$$ and any $n+1$
points drawn without repetition from any sub-collection of $n+1$
distinct balls is $n$-separated for $\omega$. That is,
\begin{equation*}\min_{\widetilde{X}\in\,\prod_{i\in I}B_i}
\MM_n(\widetilde{X})\,\geq\omega\cdot t^n,\textup{ for each set } I
\textup{ of } n+1 \textup{ distinct indices in }
\{0,\ldots,m\}.\end{equation*}

We extend this definition to $d$-regular measures in the following
way. For $x\in\Supp$ and $0<t\leq\diam(\Supp)$, we say that $\mu$ is
{\em $n$-separated in $B(x,t)$} (for $0<\delta<1$ and $\omega>0$) if
there exist $(n+2)$ balls, $\{B_i\}_{i=0}^{n+1},$ which are
$n$-separated (for $\omega$) in $B(x,t)$, centered on $\supp(\mu)$
and satisfy
$$\min_{0\leq i\leq n+1}\frac{\diam(B_i)}{2\cdot t}\geq\delta.$$
We show here that $\mu$ is $d$-separated at all scales and locations
in the following sense.
\begin{proposition}\label{proposition:rainy-may-day}There exist
$0<\delta_{\mu}=\delta_{\mu}(d,C_{\mu}) < 1$ and
$\omega_{\mu}=\omega_{\mu}(d,C_{\mu})>0$ such that for any ball
$B(x,t)\subseteq H$ with $x\in\Supp$ and $0<t\leq\diam(\Supp),$ the
following property is satisfied: There exists a $d$-separated
collection of $d+2$ balls, $\{B(x_i,\delta_{\mu}\cdot
t)\}_{i=0}^{d+1},$ contained in $B(x,t)$ as well as centered on
$\Supp$.\end{proposition}

$\acute{a}$

\subsection{Proof of
Proposition~\ref{proposition:rainy-may-day}}\label{subsection:sep-2}
For simplicity, we look  at the  ball $B(x,2\cdot t)$ (instead of
$B(x, t)$) and reduce Proposition~\ref{proposition:rainy-may-day} to
the following two parts.

{\bf Part~I:} There exist constants
$0<\delta_d=\delta_d(d,C_{\mu})\leq1/2$ and
$\omega_d=\omega_d(d,C_{\mu})>0$ such that for every $x\in\Supp$ and
 $0<t\leq\diam(\Supp),$ there is a collection of $d+1$
balls, $\{B(x_i,\delta_d\cdot t)\}_{i=0}^d$, that are $d$-separated
for $\omega_d$  in $B(x,2\cdot t)$ and whose centers
$\{x_i\}_{i=0}^d$ are in $B(x,t)\cap\Supp$.

{\bf Part~II:} Given the $d$-separated balls for $B(x,2\cdot t)$,
$\{B(x_i,\delta_d\cdot t)\}_{i=0}^d$,  constructed in Part~I, there
is a point $x_{d+1}\in B(x,t)\cap\Supp$ and constants
\begin{equation*}\label{equation:mother-of-all-deltas}0<\widetilde{\delta}_{\mu}=
\widetilde{\delta}_{\mu}(d,C_{\mu})\leq\delta_d \textup{ and
}\widetilde{\omega}_{\mu}=\widetilde{\omega}_{\mu}(d,C_{\mu})>0\end{equation*}
such that the collection of $(d+2)$ balls,
$\{B(x_i,\widetilde{\delta}_{\mu}\cdot t)\}_{i=0}^{d+1},$ is also
$d$-separated in the  ball $B(x,2\cdot t)$.

Parts~I and~II  imply the desired proposition for the ball $B(x,t)$
with $\delta_{\mu}=\widetilde{\delta}_{\mu}/2$ and
$\omega_{\mu}=\widetilde{\omega}_{\mu}/2^d>0$.

We establish Parts~I and~II in
Subsections~\ref{subsubsection:part-une}
and~\ref{subsubsection:part-deux} respectively. An elementary lemma
used in Subsection~\ref{subsubsection:part-deux} is proved
separately in Subsection~\ref{subsection:lema_sum_msr}.

\begin{remark}
The statement of Part~I is practically equivalent to the statement
of \cite[Lemma~8.2]{tolsa_riesz_rect}, which was stated without a
proof. In fact, the formulation of
\cite[Lemma~8.2]{tolsa_riesz_rect} shows how to slightly extend our
statements beyond Ahlfors regularity.
\end{remark}

\subsubsection{Part~I of the Proof}\label{subsubsection:part-une}
Our proof is inductive on $n$. If $n=1$, then let $x_0=x$ and
$\delta_0=\sqrt[d]{\frac{1}{2\cdot C_{\mu}^2}}.$ By
Lemma~\ref{lemma:$d$-regular-measure-2} we have the inequality
$$\mu\left(A\left(x_0,\delta_0\cdot t,t\right)\right)
\geq\frac{1}{2}\cdot\mu(B(x_0,t))>0.$$Then, we arbitrarily fix
$\displaystyle x_1\in A\left(x_0,\delta_0\cdot t,t\right)\cap\Supp$
and set $\displaystyle\delta_1=\delta_0\big/3$. For any
$\widetilde{x}_0\in B(x_0,\delta_1\cdot t)$ and $\widetilde{x}_1\in
B(x_1,\delta_1\cdot t),$ let
$\widetilde{X}_1=(\widetilde{x}_0,\widetilde{x}_1)$. Clearly we have
$$\MM_1(\widetilde{X}_1)=\|\widetilde{x}_0-\widetilde{x}_1\|\geq\delta_1\cdot
t,$$and thus the statement holds for $n=1$, where
$$\omega_1=\delta_1=\frac{1}{3}\cdot\left(\frac{1}{2\cdot
C_{\mu}^2}\right)^{1/d}\leq\frac{1}{2}.$$

Now, for some $1\leq n<d$, we take the induction hypothesis to be
the existence of  $n+1$ points $\{x_0,\ldots,x_n\}\subseteq
B(x,t)\cap\Supp$, and constants $0<\delta_n\leq\frac{1}{2}$ and
$\omega_n>0$ such that the collection of $n+1$ balls
$\{B(x_i,\delta_n\cdot t)\}_{i=0}^n$  is $n$-separated (for
$\omega_n$) in $B(x, 2\cdot t)$. We further assume that $x_0=x$
(which was satisfied for $n=1$).  We will construct a point
$x_{n+1}\in B(x,t)\cap\Supp$ and constants
$0<\delta_{n+1}\leq\delta_n$ and $\omega_{n+1}>0$ such that the
collection of balls $\{B(x_i,\delta_{n+1}\cdot t)\}_{i=0}^{n+1}$ is
$(n+1)$-separated (for $\omega_{n+1}$) in $B(x,2\cdot t)$ for
$\omega_{n+1}$.

For the set of balls of the induction hypothesis,
$\{B(x_i,\delta_n\cdot t)\}_{i=0}^n$, let $X_n=(x_0,\ldots,x_n)$
denote the non-degenerate simplex generated by their centers, and
furthermore let $P$ denote the orthogonal projection of $H$ onto the
$n$-plane $L[X_n]$. Let $\delta$ be an arbitrary constant with
$0<\delta\leq\delta_n\leq1/2$, where we will eventually specify a
choice for $\delta$, i.e., the constant $\delta_{n+1}$ mentioned
above.

We take an arbitrary element
\begin{equation}\label{equation:bus-time-2}\widetilde{X}_n=(\widetilde{x}_0,\ldots,\widetilde{x}_n)\in\prod_{i=0}^n
B(x_i,\delta\cdot t),\end{equation} and for such $\widetilde{X}_n$,
we note that $\{\widetilde{x}_0,\ldots,\widetilde{x}_n\}\subseteq
B\left(x_0,\frac{3}{2}\cdot t\right)$, and thus
\begin{equation}\label{equation:diameter-bound-1}\diam(\widetilde{X}_n)\leq3\cdot t.\end{equation}
Let $\widetilde{P}_{\delta}$ denote the orthogonal projection of $H$
onto the $n$-plane $L[\widetilde{X}_n]$. For convenience, we
suppress the dependence of $P$ and $\widetilde{P}_{\delta}$ on the
elements $X_n$ and $\widetilde{X}_n$ respectively.

The induction step consists of three parts. The first is the
existence of a constant $\epsilon_n>0$ (independent of $x$ and $t$)
and an element $x_{n+1}\in B(x_0,t)\cap\Supp$ such that
\begin{equation}\label{equation:uniform-distance-1}\|x_{n+1}-P\left(x_{n+1}\right)\|\geq\epsilon_n\cdot t.\end{equation}
The second part is the existence of a constant
$0<\delta_{n+1}=\delta_{n+1}(n,\delta_n,\omega_n,\epsilon_n)\leq\delta_n$
such that
\begin{equation}\label{equation:uniform-distance-2}\|x_{n+1}-\widetilde{P}_{\delta_{n+1}}\left(x_{n+1}\right)\|
\geq\frac{2\cdot\epsilon_n}{3}\cdot t.\end{equation} The last  part
of the induction proof is showing that for any
$\widetilde{x}_{n+1}\in B(x_{n+1},\delta_{n+1}\cdot t),$ we have the
lower bound
\begin{equation}\label{equation:uniform-distance-3}\|\widetilde{x}_{n+1}-
\widetilde{P}_{\delta_{n+1}}(\widetilde{x}_{n+1})\|\geq\frac{\epsilon_n}{3}\cdot
t.\end{equation}Then, we conclude the proof of part~I by combining
equation~(\ref{equation:uniform-distance-3}) with the induction
hypothesis and the product formula for contents. That is, we obtain
that for any $1\leq n\leq d$ the family of balls
$\{B(x_i,\delta_{n+1}\cdot t)\}_{i=0}^{n+1}$ is $(n+1)$-separated in
$B(x,2\cdot t)$ for the constant
\begin{equation*}\label{equation:turkey-sammich}\omega_{n+1}=\frac{\epsilon_n\cdot\omega_n}{3}.\end{equation*}

Now, to prove equation~(\ref{equation:uniform-distance-1}) for
$1\leq n<d$, let
\begin{equation*}\label{equation:nadana}\epsilon_n=\left(\frac{1}{\displaystyle2^{\frac{3\cdot
d}{2}+n+1}\cdot C_{\mu}^2}\right)^{1/(d-n)}.\end{equation*}Noting
that $\dim\left(L[X_n]\right)=n<d,$
Proposition~\ref{proposition:artube} implies
that:\begin{equation*}\label{equation:big-kahu}\mu\big(B(x,t)\setminus
\Tbe\left(L[X_n],\epsilon_n\cdot
t\right)\big)>\frac{1}{2}\cdot\mu(B(x,t))>0,\end{equation*}in
particular, $$\left[B(x,t)\cap\Supp\right]\setminus
\Tbe\left(L[X_n],\epsilon_n\cdot t\right)\not=\emptyset.$$We
arbitrarily fix $x_{n+1}\in \left[B(x,t)\cap\Supp\right]\setminus
\Tbe\left(L[x_n],\epsilon_n\cdot t\right)$, and we immediately
obtain equation~(\ref{equation:uniform-distance-1}). We also note
that
\begin{equation}\label{equation:dimaggio}\|x_{n+1}-
P(\widetilde{P}_\delta(x_{n+1}))\|\geq \epsilon_n\cdot
t.\end{equation}This follows from
equation~(\ref{equation:uniform-distance-1}) and the fact that
$P(x_{n+1})$ is the closest point to $x_{n+1}$ in the $n$-plane
$L[X_n]$.

To establish equation~(\ref{equation:uniform-distance-2}), we will
first show that there exists a constant $C_4=C_4(n,\omega_n)>0$ such
that for any $0<\delta\leq\delta_n$ we have the uniform upper bound
\begin{equation}\label{equation:uniform-distance-4}\|P(\widetilde{P}_{\delta}(y))-
\widetilde{P}_{\delta}(y)\|\leq C_4\cdot\delta\cdot t,\textup{ for
all }y\in B(x_0,t).\end{equation} Then, imposing the following
restriction on $\delta$:
\begin{equation}\label{equation:C-500}C_4\cdot\delta\leq\frac{\epsilon_n}{3},\end{equation}
and applying
equations~\eqref{equation:dimaggio}-\eqref{equation:C-500}, we
derive equation~(\ref{equation:uniform-distance-2}) as follows
%
%\begin{alignat*}{3}
%P_{d+1} \ & :&& \ \HH \rightarrow L[\widetilde{X}(d+1)]\,,\\
%P_i \ & :&& \ \HH \rightarrow L[\widetilde{X}(i;\widetilde{x}_{d+1},d+1)]\,,\\
%P_{i,d+1} \ & :&& \ \HH\rightarrow L[\widetilde{X}(i;d+1)]\,.
%\end{alignat*}
\begin{alignat*}{2}\label{equation:play-drum}
\|x_{n+1}-&\widetilde{P}_{\delta}(x_{n+1})\| \geq
\\
& \bigg|\|x_{n+1}
-P(\widetilde{P}_{\delta}(x_{n+1}))\|-\|P(\widetilde{P}_{\delta}(x_{n+1}))
-\widetilde{P}_{\delta}(x_{n+1})\|\bigg|\geq
%\\\|x_{n+1}
%-P(x_{n+1})\|-\|P(\widetilde{P}_{\delta}(x_{n+1}))
%-\widetilde{P}_{\delta}(x_{n+1})\|
\frac{2\cdot\epsilon_n}{3}\cdot t.\end{alignat*}

To establish equation~(\ref{equation:uniform-distance-4}) and
calculate the constant $C_4$, we first express the projection of any
$y\in H$  onto $L[\widetilde{X}_n]$ as
\begin{equation*}\label{equation:projected}\widetilde{P}_{\delta}(y)=\widetilde{x}_0+\sum_{i=1}^n\widetilde{s}_i(y)
\cdot\left(\widetilde{x}_i-\widetilde{x}_0\right),\end{equation*}with
$\widetilde{s}_i(y)\in\RR$, $1\leq i\leq n$. For $0\leq i\leq n$,
the points $\widetilde{x}_i$ have the  decomposition
\begin{equation}\label{equation:hack}\widetilde{x}_i=x_i+\widetilde{z}_i+\widetilde{\varepsilon}_i,\end{equation} where
$\widetilde{z}_i\in\Span\{x_1-x_0,\ldots,x_n-x_0\}$ and
$\widetilde{\varepsilon}_i$ is orthogonal to
$\Span\{x_1-x_0,\ldots,x_n-x_0\}$. Therefore, we have the following
equality for all $y\in H$:
\begin{equation}\label{equation:nasty-1}P(\widetilde{P}_{\delta}(y))-\widetilde{P}_{\delta}(y)
=-\left(\widetilde{\varepsilon}_0+\sum_{i=1}^n\widetilde{s}_i(y)\cdot
\left(\widetilde{\varepsilon}_i-\widetilde{\varepsilon}_0\right)\right).\end{equation}
Furthermore, equations~(\ref{equation:bus-time-2})
and~\eqref{equation:hack} imply the following inequality for all
$0\leq i\leq n$
\begin{equation}\label{equation:sandy-hope}\|\widetilde{\epsilon}_i\|^2\leq
\|\widetilde{z}_i\|^2+\|\widetilde{\varepsilon}_i\|^2\,\leq\,(\delta\cdot
t)^2.\end{equation} Thus, applying
equation~(\ref{equation:sandy-hope}) and the triangle inequality  to
the RHS of equation~(\ref{equation:nasty-1}), we have that
\begin{equation}\label{equation:janet-jackson}\|P(\widetilde{P}_{\delta}(y))
-\widetilde{P}_{\delta}(y)\|\leq\left(1+\sum_{i=1}^n2\cdot|s_i(y)|\right)\cdot\delta\cdot
t.\end{equation}

Now, to bound the RHS of equation~(\ref{equation:janet-jackson}) for
$y\in B(x_0,t)$  (and thereby calculate an upper bound for the
constant $C_4$ of equation~(\ref{equation:uniform-distance-4})), we
calculate a uniform bound for the quantities $\{|s_i(y)|\}_{i=1}^n$.
In fact, we will establish the following inequality for all $y\in
B(x_0,t)$:\begin{equation}\label{equation:cedar}\max_{1\leq i\leq
n}|\widetilde{s}_i(y)|\leq \frac{2\cdot
3^{n-1}}{\omega_n}.\end{equation}The combination of such a bound
with equation~(\ref{equation:janet-jackson}) clearly implies
equation~(\ref{equation:uniform-distance-4}), where
\begin{equation}\label{equation:bus-time-3}C_4=\left(1+n\cdot\frac{4\cdot3^{n-1}}{\omega_n}\right).\end{equation}

We first note that the coefficients $\widetilde{s}_i(y)$ satisfy the
following equation for all $1\leq i\leq n$:
\begin{equation}\label{equation:formula-1}\sin(\theta_i(\widetilde{X}_n))\cdot
|\widetilde{s}_i(y)|\cdot \|\widetilde{x}_i-\widetilde{x}_0\|
=\dist(\widetilde{P}_{\delta}(y),L[\widetilde{X}_n(i)]).\end{equation}Obtaining
an upper bound on the RHS of equation~(\ref{equation:formula-1}) as
well as a lower bound on the quantity
$\sin(\theta_i(\widetilde{X}_n))\cdot
\|\widetilde{x}_i-\widetilde{x}_0\|,$ will then establish
equation~(\ref{equation:cedar}).

We determine an upper bound by noting that $\widetilde{x}_0\in
B(x_0,t)\cap L[\widetilde{X}(i)],$ and thus for any $y\in B(x_0,t)$
\begin{equation}\label{equation:chores}\dist(\widetilde{P}_{\delta}(y),L[\widetilde{X}_n(i)])\leq
\left\|\widetilde{P}_{\delta}(y)-\widetilde{x}_0\right\|\leq\|y-\widetilde{x}_0\|\leq
2\cdot t.\end{equation}

In order to obtain the lower bound,  we apply the product formula
for contents as well as equation~(\ref{equation:diameter-bound-1}),
and get that for any $0<\delta\leq\delta_n$ and all $1\leq i\leq n$
\begin{equation*}\label{equation:little-dream}\MM_n(\widetilde{X}_n)=
\sin(\theta_i(\widetilde{X}_n))\cdot\|\widetilde{x}_i-
\widetilde{x}_0\|\cdot
\MM_{n-1}(\widetilde{X}_n(i))\leq\sin(\theta_i(\widetilde{X}_n))\cdot
\|\widetilde{x}_i-\widetilde{x}_0\|\cdot3^{n-1}\cdot
t^{n-1}.\end{equation*}
%where $\theta_i\left(\widetilde{X}_n\right)$
%is the elevation angle of $\widetilde{x}_i-\widetilde{x}_0$ with
%respect to the subspace $V\left[\widetilde{X}_n(i)\right]$.
Combining this with the induction hypothesis, i.e.,
$\MM_n(\widetilde{X}_n)\geq\omega_n\cdot t^n$, we obtain the
inequality
\begin{equation}\label{equation:calhoun}\min_{1\leq i\leq n}\
\sin(\theta_i(\widetilde{X}_n))\cdot\|\widetilde{x}_i-\widetilde{x}_0\|
\,\geq \,\frac{\omega_n}{3^{n-1}}\cdot t.\end{equation} Applying the
bounds of equations~(\ref{equation:calhoun})
and~(\ref{equation:chores}) to equation~(\ref{equation:formula-1}),
we conclude equation~(\ref{equation:cedar}), and consequently
equations~\eqref{equation:uniform-distance-4}
and~\eqref{equation:bus-time-3}. We note that the constant
$\delta=\delta_{n+1}$ needs to satisfy
equation~(\ref{equation:C-500}) and the requirement
$0<\delta_{n+1}\leq \delta_n$.  We thus set its value in the
following way:
\begin{equation}\label{equation:dog-days}\delta_{n+1}=\min\left\{\frac{\epsilon_n}{3\left(1+n\cdot\frac{4\cdot3^{n-1}
}{\omega_n}\right)}\,,\,\delta_n\right\}.\end{equation}

To prove the final part of the induction argument, i.e.,
equation~(\ref{equation:uniform-distance-3}), we apply the triangle
inequality and equations~\eqref{equation:bus-time-2} (with
$\delta=\delta_{n+1}$),~(\ref{equation:uniform-distance-2})
and~(\ref{equation:dog-days}), obtaining that for any
$\widetilde{x}_{n+1}\in B\left(x_{n+1},\delta_{n+1}\cdot t\right)$
\begin{multline*}\label{equation:last-of-mohawk}\left\|\widetilde{x}_{n+1}-\widetilde{P}_{\delta_{n+1}}
\left(\widetilde{x}_{n+1}\right)\right\|\geq\Bigg|\left\|x_{n+1}
-\widetilde{P}_{\delta_{n+1}}(\widetilde{x}_{n+1})\right\|-\left\|\widetilde{x}_{n+1}-x_{n+1}\right\|\Bigg|\geq\\\left\|x_{n+1}
-\widetilde{P}_{\delta_{n+1}}(x_{n+1})\right\|-\left\|\widetilde{x}_{n+1}-x_{n+1}\right\|\geq\frac{\epsilon_n}{3}\cdot
t.\end{multline*}

\subsubsection{Part~II of Proof}\label{subsubsection:part-deux}
Using the set of $d$-separated balls of Part~I,
$\{B(x_i,\delta_d\cdot t)\}_{i=0}^d$,  we take the element
$X_d=(x_0,\ldots,x_d)$, and for $0<\rho<1$ we define the constant
\begin{equation*}\epsilon_{\rho}=\frac{1-\rho}{2^{\frac{5\cdot d}{2}-1}\cdot
C_{\mu}^2}.\end{equation*}
We note that by
Proposition~\ref{proposition:artube}
\begin{equation}\label{equation:no-lefty}\min_{0\leq i\leq
d}\,\mu\Big(B(x,t)\setminus\Tbe\left(L[X_d(i)], \epsilon_{\rho}\cdot
t\right)\Big)\geq\rho\cdot\mu\left(B(x,t)\right).\end{equation}
Hence, imposing the restriction $\rho>d/(d+1)$, and applying
Lemma~\ref{lemma:inter-measure-1} (presented in
Subsection~\ref{subsection:lema_sum_msr} below) with $\nu$ being the
restricted and scaled measure $\mu|_{B(x,t)}/\mu(B(x,t))$, $\xi =
\rho$, $A_i = B(x,t)\setminus\Tbe (L[X_d(i)], \epsilon_{\rho}\cdot
t)$ for $0\leq i\leq d$, and $k=d$, we get the following lower
bound:
\begin{equation}\label{equation:no-lefty-1}\mu\left(B(x,t)\setminus\bigcup_{i=0}^d\Tbe\left(L[X_d(i)],\epsilon_{\rho}\cdot
t\right)\right)>0.\end{equation}Therefore, for such $\rho$ there
exists a point $x_{d+1}\in B(x,t)\cap\Supp$ so that
\begin{equation}\min_{0\leq i\leq d}\,\DIst\left(x_{d+1},L[X_d(i)]\right)\ >\epsilon_{\rho}\cdot t.\end{equation}

To choose the constants
$\widetilde{\delta}_{\mu}=\widetilde{\delta}_{\mu}(d,C_{\mu})>0$ and
$\widetilde{\omega}_{\mu}=\widetilde{\omega}_{\mu}(d,C_{\mu})>0,$ as
well as verify the claim of $d$-separation, we use practically the
same arguments as those for proving
equations~(\ref{equation:uniform-distance-1})-(\ref{equation:uniform-distance-3}).
We arbitrarily fix $0<\delta\leq\delta_d$, while later specifying
its value, and an element
$$\widetilde{X}_d=(\widetilde{x}_0,\ldots,\widetilde{x}_d)\in\prod_{i=0}^d B(x_i,\delta\cdot t).$$ By
the conclusion of Part~I of the proof,  we have that
$$\MM_d(\widetilde{X}_d)\geq\omega_d\cdot
t^d. $$Furthermore, $\diam(\widetilde{X}_d)\leq 3\cdot t.$ Combining
these with the product formula for contents, we obtain the
inequality
\begin{equation}\label{equation:0-content-use}\min_{0\leq i\leq d}\MM_{d-1}(\widetilde{X}_d(i))\geq\frac{\omega_d}{3}\cdot
t^{d-1}.\end{equation}

For $0\leq i\leq d$, let $P_i$ and $\widetilde{P}_{\delta,i}$ denote
the orthogonal projections of $H$ onto $L[X_d(i)]$ and
$L[\widetilde{X}_d(i)],$ respectively. By virtually the same
argument producing equation~(\ref{equation:uniform-distance-4}),
while applying equation~\eqref{equation:0-content-use}, we have that
for all $y\in B(x,t)$,\begin{equation*}\max_{0\leq i\leq
d}\,\left\|P_i\left(\widetilde{P}_{\delta,i}(y)\right)-\widetilde{P}_{\delta,i}(y)\right\|\,\leq\,\left(1+(d-1)\cdot\frac{4\cdot
3^{d-1}}{\omega_d}\right)\cdot\delta\cdot t\,.\end{equation*}Next,
we impose the further restriction $\rho_0=\frac{d+0.5}{d+1},$ and
for this value of $\rho$ we set
\begin{equation*}\label{equation:don't-cry}\widetilde{\delta}_{\mu}=\min\left\{\frac{\epsilon_{\rho_0}}{3\cdot\left(1+(d-1)\cdot\frac{4\cdot
3^{d-1}}{\omega_d}\right)},\,\delta_d\right\}.\end{equation*}By the
same calculations producing
equation~(\ref{equation:uniform-distance-3}), we get that
\begin{equation}\label{equation:ibm}\min_{0\leq i\leq d}\,\|\widetilde{x}_{d+1}-\widetilde{P}_{\widetilde{\delta}_{\mu},i}(\widetilde{x}_{d+1})\|
\geq\frac{\epsilon_{\rho_0}}{3}\cdot t\ \textup{ for all
}\widetilde{x}_{d+1}\in B(x_{d+1},\widetilde{\delta}_{\mu}\cdot
t).\end{equation}

Finally, combining Part~I and
equations~(\ref{equation:0-content-use}) and~(\ref{equation:ibm})
along with the product formula for contents, we have that for any
$$\widetilde{X}_{d+1}=(\widetilde{x}_0,\ldots,\widetilde{x}_{d+1})\in\prod_{i=0}^{d+1}B(x_i,\delta\cdot t),$$the following
inequality is satisfied\begin{equation*}\min_{0\leq i\leq d+1}\
\MM_d\left(\widetilde{X}_{d+1}(i)\right)\,\geq\,\frac{\omega_d\cdot\epsilon_{\rho_0}}{9}\cdot
t^d.\end{equation*}Therefore, taking
\begin{equation*}\label{equation:mother-of-all-lambdas}\widetilde{\omega}_{\mu}=\frac{\epsilon_{\rho_0}\cdot\omega_d}{9},\end{equation*}
the collection of balls $\{B(x_i,\widetilde{\delta}_{\mu}\cdot
t)\}_{i=0}^{d+1}$ is $d$-separated in $B(x,2\cdot t)$ for
$\widetilde{\omega}_{\mu}$.\qed

\subsubsection{An Elementary Lemma}
\label{subsection:lema_sum_msr} We establish the following
elementary proposition which was used in
Subsection~\ref{subsubsection:part-deux} and will also be used later
in Subsection~\ref{subsection:lemma-1}.
\begin{lemma}\label{lemma:inter-measure-1} If $\nu$ is a
Borel probability measure, $A_0, A_1, \ldots, A_d$ are measurable
sets (w.r.t.~$\nu$), $0<\xi<1$, and
\begin{equation}\label{equation:didng}\min_{0\leq i\leq
d}\nu(A_i)\geq\xi,\end{equation}
then for any $0\leq
k\leq d$ the following inequality holds
\begin{equation}\label{equation:didng-1}\nu\left(\bigcap_{i=0}^k
A_i\right)\geq (k+1)\cdot\xi-k \,.\end{equation}\end{lemma}

\begin{proof}The proof is by induction.
Equation~(\ref{equation:didng}) clearly implies the inequality of
equation~(\ref{equation:didng-1}) when $k=0$. Supposing that
equation~\eqref{equation:didng-1} holds for some $0\leq k< d$, we
note that
\begin{equation}\label{equation:app-meas-1} 1 \geq\nu\left(\bigcap_{i=0}^k A_i\cup
A_{k+1}\right)=\nu\left(\bigcap_{i=0}^k A_i\right)
+\nu(A_{k+1})-\nu\left(\bigcap_{i=0}^{k+1}A_i\right).\end{equation}
Thus, by the induction hypothesis and
equation~\eqref{equation:app-meas-1} we have that
\begin{multline*}\nu\left(\bigcap_{i=0}^{k+1}A_i\right)\geq\nu\left(\bigcap_{i=0}^k
A_i\right)+\nu(A_{k+1})-1 \geq
(k+1)\cdot\xi-k+\xi-1 = (k+2) \cdot \xi-(k+1) \,.\end{multline*}
%
%Therefore equation~\eqref{equation:didng-1} holds by induction for
%all $0\leq k\leq d$.
\end{proof}

\section{The  Proof of Theorem~\ref{theorem:doozy}}\label{section:doozy-1}
In order to prove Theorem~\ref{theorem:doozy}, we will establish the
existence of constants $\lambda_0=\lambda_0(d,C_{\mu})$ and
$C_1=C_1(d,C_{\mu})$ such that there exists a $d$-plane $L_{(x,t)}$
with
\begin{equation}\label{equation:happy-times}\int_{B(x,t)}\left(\frac{\DIst\left(y,L_{(x,t)}\right)}
{2\cdot t}\right)^2\ud\mu(y)\leq C_1\cdot
c_{\mathrm{MT}}^2(x,t,\lambda_0),\end{equation}for any $x\in\Supp$
and $0<t\leq\diam(\Supp)$.  Applying the definition of the $\beta_2$
numbers to equation~(\ref{equation:happy-times}) then proves
Theorem~\ref{theorem:doozy}.

Our approach for establishing equation~(\ref{equation:happy-times})
generalizes the proof of L{\'e}ger~\cite[Lemma 2]{Leger} for the
case $d=1$. In that case, constructing the line $L_{(x,t)}$ is
relatively straightforward and short. However, for $d\geq2$  there
are combinatorial and geometric issues that do not manifest
themselves when $d=1$, e.g., the proofs of
Proposition~\ref{proposition:sappy-homilies} and
Lemma~\ref{lemma:lemma-1} below, and the notion of $d$-separation
for $d\geq2$ (Section~\ref{section:geom-prop} above). We present the
overall argument in Subsection~\ref{subsection:big-proof}, and we
leave the details   to Subsections~\ref{subsection:lemma-1}
and~\ref{subsection:unwanted-2}. Preliminary notation and
observations are provided in Subsection~\ref{subsection:funky-slam}.
\subsection{Notation and Preliminary
Observations}\label{subsection:funky-slam}For any $x\in\Supp$,
$0<t\leq\diam(\Supp)$, $0<\lambda<2$, and $0\leq i<j\leq d+1$ we
define the following slices of the set $U_{\lambda}(B(x,t))$ of
equation~(\ref{equation:big-U}):
\begin{equation*}U_{\lambda}\left(x,t\,\big|\,X(i)\right)=\{y\in
B(x, t):X(y,i)\in U_{\lambda}(B(x,t))=\},\end{equation*}
\begin{equation*}U_{\lambda}\left(x,t\,\big|\,X(i;j)\right)=\left\{(y,z)\in
B(x, t)^2:X(y,i;z,j)\in U_{\lambda}(B(x,t))=\right\},\end{equation*}
\begin{equation*}U_{\lambda}\left(x,t\,\big|\,X(y,i;j)\right)=\left\{z\in
B(x, t):X(y,i;z,j)\in
U_{\lambda}(B(x,t))=\right\},\end{equation*}\begin{equation*}U_{\lambda}\left(x,t\,\big|\,X(i;y,j)\right)=\left\{z\in
B(x, t):X(z,i;y,j)\in U_{\lambda}(B(x,t))=\right\}.\end{equation*}%Each
%of these sets can be empty depending on the elements $X(i)$,
%$X(i,j),$ and $X(y,i)$, which were  defined in
%equations~(\ref{equation:removal})-(\ref{equation:double-remove-single-add})
%(see also Remark~\ref{remark:notation-convention}).
In addition, we
fix the following constant of $1$-separation
\begin{equation}\label{equation:lambda-mo}\lambda_0=\frac{\delta_{\mu}}{2}\,,\end{equation}
where $\delta_{\mu}$ is the constant suggested by
Proposition~\ref{proposition:rainy-may-day}.

For the remainder of the proof (i.e., the whole section) we
arbitrarily fix $x\in\Supp$ and $0<t\leq\diam(\Supp)$, and some
$d$-separated collection of balls $\{B(x_i,\delta_{\mu}\cdot
t)\}_{i=0}^{d+1}$ in $B(x,t)$ for the constant $\omega_{\mu}$ (see
Proposition~\ref{proposition:rainy-may-day}). We denote
$B_i=B(x_i,\delta_{\mu}\cdot t)$ for $0\leq i\leq d+1$. Restricting
our attention to only the first $(d+1)$ balls,  we also form an
arbitrary element
$$\widetilde{X}(d+1)=(\widetilde{x}_0,\ldots,\widetilde{x}_d)\in\prod_{i=0}^d \frac{1}{2}\cdot B_i.$$

We note
that\begin{equation}\label{equation:best-inclusion}B(x,t)\setminus\bigcup_{i=0}^d
B_i\  \subseteq\
U_{\lambda_0}(x,t|\widetilde{X}(d+1))\end{equation}and
\begin{equation}\label{equation:no-gut}B_i\nsubseteq
U_{\lambda_0}(x,t|\widetilde{X}(d+1)),\textup{ for each }0\leq i\leq
d.\end{equation}

\subsection{The Essence of the Proof of Theorem~\ref{theorem:doozy}}\label{subsection:big-proof}
For $0<\rho<\infty$ let
\begin{multline}\label{equation:fish-set}\mathcal{E}(\rho)=\Bigg\{\widetilde{X}(d+1)\in\prod_{i=0}^d\frac{1}{2}
\cdot
B_i:\\\int_{U_{\lambda_0}\left(x,t\,\big|\widetilde{X}(d+1)\right)}
\frac{\pds_{\widetilde{x}_0}^2\left(\widetilde{X}(y,d+1)\right)}
{\diam\left(\widetilde{X}(y,d+1)\right)^{d(d+1)}}\,\ud\mu(y)\leq
\rho\cdot
\frac{c_{\mathrm{MT}}^2(x,t,\lambda_0)}{t^{d(d+1)}}\Bigg\}.\end{multline}We
will show that $\mu^{d+1}(\mathcal{E}(\rho))$ is sufficiently large
for some $0<\rho<\infty$.

First, applying Chebychev's inequality to
equation~\eqref{equation:fish-set} we obtain that
\begin{equation}\label{equation:mighty-mince}\mu^{d+1}\left(\prod_{i=0}^d \frac{1}{2}\cdot B_i\setminus
\mathcal{E}(\rho)\right)\leq\frac{t^{d(d+1)}}{\rho}\,.\end{equation}
Next, we note that the $d$-regularity of $\mu$ implies that
\begin{equation}\label{equation:bigger-is-better}
%\leq
\mu^{d+1}\left(\prod_{i=0}^d\frac{1}{2}\cdot B_i\right) \geq
\frac{1}{C_{\mu}^{d+1}}\cdot\left(\lambda_0\cdot t\right)^{d(d+1)}
%\leq
%C_{\mu}^{d+1}\cdot\left(\lambda_0\cdot
%t\right)^{d(d+1)}
.\end{equation}Thus, combining
equations~\eqref{equation:mighty-mince}
and~\eqref{equation:bigger-is-better}, and taking
\begin{equation}\label{equation:fish-rho}\rho_1=\rho_1(d,C_{\mu})>\frac{2}{\lambda_0^{d(d+1)}}\cdot
C_{\mu}^{d+1},\end{equation}we obtain the lower
bound\begin{equation}\label{equation:first-bit-o-measure}\mu^{d+1}\left(\mathcal{E}(\rho_1)\right)>\frac{1}{2}
\cdot\mu^{d+1}\left(\prod_{i=0}^d\frac{1}{2}\cdot
B_i\right)>0.\end{equation}

We will show that the desired $d$-plane, $L_{(x,t)}$, of
equation~(\ref{equation:happy-times}) is obtained by
$L[\widetilde{X}(d+1)]$ for some $\widetilde{X}(d+1)\in
\mathcal{E}(\rho_1)$. In fact, for any such $\widetilde{X}(d+1)$ we
immediately obtain control on a part of the integral on the LHS of
equation~(\ref{equation:happy-times}) as follows. Since
$\widetilde{X}(d+1)\in \mathcal{E}(\rho_1)$ is $d$-separated, by
Lemma~\ref{lemma:hey-diddle-diddle}  the following lower bound holds
for  all $y\in H$
\begin{equation}\label{equation:matador}\frac{\displaystyle
\pds_{\widetilde{x}_0}^2\left(\widetilde{X}(y,d+1)
\right)}{\diam\left(\widetilde{X}(y,d+1)\right)^{d(d+1)}}\geq
\frac{\omega_{\mu}^2}{2^{(d+1)(d+2)}}\cdot\left(\frac{\DIst
(y,L[\widetilde{X}(d+1)])}
{t}\right)^2\cdot\frac{1}{t^{d(d+1)}}\,.\end{equation}
Thus, by
equations~(\ref{equation:fish-set}) and~(\ref{equation:matador}) we
have that  for any $\widetilde{X}(d+1)\in \mathcal{E}(\rho_1)$
\begin{equation}\label{equation:krab}\int_{U_{\lambda_0}\left(x,t\,
\big|\widetilde{X}(d+1)\right)}
\left(\frac{\DIst(y,L[\widetilde{X}(d+1)])}{t}\right)^2\di\mu(y)\
\leq \frac{2^{(d+1)(d+2)}}{\omega_{\mu}^2}\cdot\rho_1\cdot\
c_{\mathrm{MT}}^2(x,t,\lambda_0).\end{equation} Combining this with
the set inclusion of equation~(\ref{equation:best-inclusion})
implies that
\begin{equation}\label{equation:almost-but-no}\int_{B(x,t)\setminus\bigcup_{i=0}^d B_i}
\left(\frac{\DIst(y,L[\widetilde{X}(d+1)])}{t}\right)^2\di\mu(y)\
\leq\frac{2^{(d+1)(d+2)}}{\omega_{\mu}^2}\cdot\rho_1\cdot
c_{\mathrm{MT}}^2(x,t,\lambda_0).\end{equation}

Despite the upper bound of equation~(\ref{equation:almost-but-no}),
the condition of the set $\mathcal{E}(\rho_1)$ does not help us to
obtain a bound for the integrals over the individual balls $B_i$,
$0\leq i\leq d$. This incompleteness follows from
equation~(\ref{equation:no-gut}).  In order to obtain such an upper
bound (thus concluding equation~(\ref{equation:happy-times})), we
must impose further restrictions on the element
$\widetilde{X}(d+1)$.

For $0<\rho<\infty$, let
\begin{multline}\label{equation:math-cal-a}\mathcal{A}(\rho)=\Bigg\{\widetilde{X}(d+1)\in\prod_{i=0}^d\frac{1}{2}\cdot B_i:\\\max_{0\leq i\leq d}
\ \int_{U_{\lambda_0}\left(x,t\,\big|\widetilde{X}(i;d+1)
\right)}\frac{\pds_{\left(\widetilde{X}(y,i;z,d+1)\right)_0}^2\big(\widetilde{X}(y,i;z,d+1)\big)}
{\diam\left(\widetilde{X}(y,i;z,d+1)\right)^{d(d+1)}}\
\ud\mu^2(y,z)\leq\frac{\rho\cdot
c_{\mathrm{MT}}^2(x,t,\lambda_0)}{t^{d^2}}\Bigg\}.\end{multline}Below
in Subsection~\ref{subsection:lemma-1} we prove the following
lemma.\begin{lemma}\label{lemma:lemma-1}There exists a constant
$\rho_2=\rho_2(d,C_{\mu})>0$ such
that$$\mu^{d+1}\left(\mathcal{A}(\rho_2)\right)>\frac{1}{2}\cdot\mu^{d+1}\left(\prod_{i=0}^d\frac{1}{2}\cdot
B_i\right)>0,$$for any $x\in\Supp$ and
$0<t\leq\diam(\Supp)$.\end{lemma}The condition imposed by
$\mathcal{A}(\rho_2)$   yields the following estimate which is
proved in
Subsection~\ref{subsection:unwanted-2}.\begin{proposition}\label{proposition:sappy-homilies}There
exists a constant $C_5=C_5(d,C_{\mu})$ such that for any element
$\widetilde{X}(d+1)\in \mathcal{A}(\rho_2)$:
$$\max_{0\leq i\leq d}\ \int_{B_i}\left(\frac{\dist(y,L[\widetilde{X}(d+1)])}
{t}\right)^2\di\mu(y)\leq C_5\cdot
c_{\mathrm{MT}}^2(x,t,\lambda_0).$$\end{proposition}

Finally,  equation~\eqref{equation:first-bit-o-measure} and
Lemma~\ref{lemma:lemma-1} imply that $\mathcal{E}(\rho_1)\cap
\mathcal{A}(\rho_2)\not=\emptyset$. Thus, fixing an arbitrary
$\widetilde{X}(d+1)\in \mathcal{E}(\rho_1)\cap \mathcal{A}(\rho_2)$,
equation~(\ref{equation:happy-times}) is deduced from
equation~(\ref{equation:almost-but-no}) and
Proposition~\ref{proposition:sappy-homilies}.\qed

\subsection{The Proof of
Lemma~\ref{lemma:lemma-1}}\label{subsection:lemma-1} For each $0\leq
i\leq d$ we define the following cartesian product
\begin{equation*}\label{equation:almost-big}A_i=\prod_{0\leq j\not=i\leq d}\frac{1}{2}\cdot
B_j.\end{equation*}
We note that the $d$-regularity of $\mu$ and
equation~\eqref{equation:lambda-mo} trivially imply the following
estimate for each $0\leq i\leq d$:
\begin{equation}\label{equation:estimates-die}\mu^d(A_i) \geq \frac{1}{C^d_{\mu}}\cdot(\lambda_0\cdot
t)^{d^2}.
%\leq C_{\mu}^d\cdot(\lambda_0\cdot
%t)^{d^2}.
\end{equation}
Then, for $0<\rho<\infty$ and $0\leq i\leq
d$, we define the
set\begin{multline}\label{equation:symmetry-sets}\mathcal{A}_i(\rho)=\bigg\{\widetilde{X}(i;d+1)\in
A_i:\\\int_{U_{\lambda_0}\left(x,t\,\big|\widetilde{X}(i;d+1)
\right)}\frac{\pds_{\left(\widetilde{X}(y,i;z,d+1)\right)_0}^2\big(\widetilde{X}(y,i;z,d+1)\big)}{\diam\left(\widetilde{X}(y,i;z,d+1)\right)^{d(d+1)}}\
\ud\mu^2(y,z)\leq\frac{\rho\cdot
c_{\mathrm{MT}}^2(x,t,\lambda_0)}{t^{d^2}}\bigg\},\end{multline}and
we embed it in the product $\prod_{j=0}^d\frac{1}{2}\cdot B_j$ by
defining the set
\begin{equation*}\label{equation:augmentale}\underline{\mathcal{A}_i\left(\rho\right)}=\Big\{\widetilde{X}(y,i;d+1):\widetilde{X}(i,d+1)\in
\mathcal{A}_i\left(\rho\right)\textup{ and }y\in \frac{1}{2}\cdot
B_i\Big\}.\end{equation*}From this definition we see that
\begin{equation}\label{equation:killing-moi}\mu^{d+1}\left(\underline{\mathcal{A}_i(\rho)}\right)
=\mu^d(\mathcal{A}_i(\rho))\cdot\mu\left(\frac{1}{2}\cdot
B_i\right).\end{equation}Furthermore, for the set
$\mathcal{A}(\rho)$ of equation~(\ref{equation:math-cal-a}), we note
the inclusion
\begin{equation}\label{equation:inter-all-sym}\bigcap_{i=0}^d\underline{\mathcal{A}_i(\rho)}
\subseteq \mathcal{A}(\rho)\subseteq \prod_{i=0}^d\frac{1}{2}\cdot
B_i.\end{equation}

We next find $\rho$ such that $\mu^{d+1}(\mathcal{A}(\rho))$ is
sufficiently large. We do this by first using
equation~\eqref{equation:killing-moi} to find $\rho$ such that the
individual $\mu^{d+1}(\underline{\mathcal{A}_i(\rho)})$, $0\leq
i\leq d$, are sufficiently large, and then applying
equation~\eqref{equation:inter-all-sym} to get the desired
conclusion about $\mathcal{A}(\rho)$.

Applying Chebychev's inequality to
equation~\eqref{equation:symmetry-sets} implies that for all $0 \leq
i \leq d$
\begin{equation}\label{equation:deadly-force}\mu^d\left(\mathcal{A}_i(\rho)\right)\geq\mu^d
\left(A_i\right)-\frac{t^{d^2}}{\rho}.\end{equation}In order to
choose $\rho$, for any $0<\xi<1$ we define
\begin{equation*}\label{equation:rhino-2}\rho(\xi)=\frac{C_{\mu}^d}{1-\xi}\cdot\left(\frac{1}
{\lambda_0}\right)^{d^2},\end{equation*}and by applying the
estimates of equations~(\ref{equation:estimates-die})
and~(\ref{equation:deadly-force}) we obtain that
\begin{equation}\label{equation:for-all-i-please}\mu^d\left(\mathcal{A}_i(\rho(\xi))\right)
\geq\xi\cdot\mu^d\left(A_i\right),\textup{ for each }0\leq i\leq
d.\end{equation}Hence, by equations~(\ref{equation:killing-moi})
and~(\ref{equation:for-all-i-please}) we have the lower bound
\begin{equation}\mu^{d+1}\left(\underline{\mathcal{A}_i(\rho(\xi))}\right)
\geq\xi\cdot\mu^{d+1}\left(\prod_{j=0}^d \frac{1}{2}\cdot
B_j\right),\textup{ for all }0\leq i\leq d.\end{equation}

Therefore, letting $\rho_2=\rho_2(\xi)$ where
\begin{equation}\label{equation:xi-is-where}\xi>\frac{d+1/2}{d+1}\
,\end{equation}
and applying Lemma~\ref{lemma:inter-measure-1} (with $\nu$ being the
measure $\mu^{d+1}$ restricted to the set $\prod_{j=0}^d 1/2 \cdot
B_j$ and scaled to 1 on that set, $A_i =
\underline{\mathcal{A}_i(\rho(\xi))}$ for $0\leq i\leq d$, and
$k=d$) we get the following lower bound:
\begin{equation}\label{equation:haploid}\mu^{d+1}\left(\bigcap_{i=0}^d\
\underline{\mathcal{A}_i(\rho_2(\xi))}\right)\ \geq\
\big((d+1)\cdot\xi-d\big)\cdot\mu^{d+1}\left(\prod_{j=0}^d
\frac{1}{2}\cdot B_j\right)\ >
\frac{1}{2}\cdot\mu^{d+1}\left(\prod_{j=0}^d \frac{1}{2}\cdot
B_j\right).\end{equation}Lemma~\ref{lemma:lemma-1} thus follows from
equations~\eqref{equation:inter-all-sym}
and~\eqref{equation:haploid}.\qed

\subsection{The Proof of
Proposition~\ref{proposition:sappy-homilies}}\label{subsection:unwanted-2}
Up until this point, we have not used the full statement of
Proposition~\ref{proposition:rainy-may-day}. We have only used the
first $(d+1)$ balls, $B_0,\ldots,B_d$,  in the definitions of the
sets $\mathcal{E}(\rho_1)$ and $\mathcal{A}(\rho_2)$, and we have
completely ignored the $(d+2)$-nd ball of the $d$-separated
collection $\{B_j\}_{j=0}^{d+1}.$  The proof of
Proposition~\ref{proposition:sappy-homilies} requires the use of
this final ball, which we have denoted by $B_{d+1}$.  We use this
ball to formulate the following lemma (whose proof is given in
Subsection~\ref{subsection:proof_lema_no_other}).
\begin{lemma}\label{lemma:there-is-no-other}There exist constants $C_6=C_6(d,C_{\mu},\lambda_0)$ and $C_7=C_7(d,C_{\mu},\lambda_0)$
such that for any fixed
$\widetilde{X}(d+1)\in\mathcal{E}(\rho_1)\cap\mathcal{A}(\rho_2)$
and fixed $0\leq i\leq d$, the following property is satisfied:
There exists a point$$\widetilde{x}_{d+1}\in \frac{1}{2}\cdot
B_{d+1}\cap\Supp$$with
\begin{equation}\label{equation:hot-tin-roof}\int_{U_{\lambda_0}\left(x,t\,\big|\widetilde{X}(i;\,\widetilde{x}_{d+1},d+1)\right)}
\frac{\pds_{\left(\widetilde{X}\left(y,i;\widetilde{x}_{d+1},d+1\right)\right)_0}^2\left(\widetilde{X}\left(y,i;\widetilde{x}_{d+1},d+1\right)\right)}
{\diam\left(\widetilde{X}\left(y,i;\widetilde{x}_{d+1},d+1\right)\right)^{d(d+1)}}\,\ud\mu(y)\leq
C_6\cdot\frac{c_{\mathrm{MT}}^2(x,t,\lambda_0)}{t^{d^2+d}},\end{equation}
and\begin{equation}\label{equation:dog-day-after}\left(\frac{\DIst(\widetilde{x}_{d+1},L[\widetilde{X}(d+1)
])}{t}\right)^2\leq C_7\cdot
\frac{c_{\mathrm{MT}}^2(x,t,\lambda_0)}{t^d}.\end{equation}
\end{lemma}
We will prove this lemma in
Subsection~\ref{subsection:proof_lema_no_other}, and will then use
it in Subsection~\ref{subsubsection:lookyling} to prove
Proposition~\ref{proposition:sappy-homilies}.

\subsubsection{Proof of
Lemma~\ref{lemma:there-is-no-other}}\label{subsection:proof_lema_no_other}
To construct the point $\widetilde{x}_{d+1}$, for any fixed
$\widetilde{X}(d+1)\in\mathcal{A}(\rho_2)$ and any $0\leq i\leq d$
we define the following two sets for $\rho_2$ of
Lemma~\ref{lemma:lemma-1} and any $0<\tau<\infty$:
\begin{multline*}\mathcal{Q}(\tau)=\Bigg\{z\in \frac{1}{2}\cdot
B_{d+1}:\int_{U_{\lambda_0}\left(x,t\,\big|\widetilde{X}(i;\,z,d+1)\right)}
\frac{\pds_{\left(\widetilde{X}\left(y,i;z,d+1\right)\right)_0}^2
\left(\widetilde{X}\left(y,i;z,d+1\right)\right)}
{\diam\left(\widetilde{X}\left(y,i;z,d+1\right)\right)^{d(d+1)}}\,\ud\mu(y)\leq\\
\frac{\tau}{t^d}\cdot\rho_2\cdot\frac{c_{\mathrm{MT}}^2(x,t,\lambda_0)}
{t^{d^2}}\Bigg\},\end{multline*}and\begin{equation*}\mathcal{G}(\tau)=\left\{z\in
\frac{1}{2}\cdot B_{d+1}:\left(\frac{\DIst
(z,L[\widetilde{X}(d+1)])}{t}\right)^2\leq \frac{\tau}{t^d}\cdot
\rho_1\cdot\frac{2^{(d+1)(d+2)}}{\omega_{\mu}^2}\cdot
c_{\mathrm{MT}}^2(x,t,\lambda_0)\right\}.\end{equation*} The idea is
to find a large enough $\tau$ so that the intersection of these two
sets is non-empty.

We first focus on the set $\mathcal{Q}(\tau)$ and specify a value
for $\tau$ such that $\mu(\mathcal{Q}(\tau))$ is sufficiently large.
%Trivially, we have the inclusion
%$$\frac{1}{2}\cdot B_{d+1}\subseteq
%U_{\lambda_0}\left(x,t\big|\widetilde{X}(d+1)\right),$$
%and thus
%\begin{equation}\label{equation:not-funny}\widetilde{X}(y,i;z,d+1)\in
%U_{\lambda_0}(x,t),
%\end{equation}for all $z\in\frac{1}{2}\cdot B_{d+1}$ and $y\in
%U_{\lambda_0}\left(x,t\,\big|\widetilde{X}(i;\,z,d+1)\right)$.
Since $\widetilde{X}(d+1)\in\mathcal{A}(\rho_2)$, by
equation~\eqref{equation:math-cal-a} % and~(\ref{equation:not-funny})
we get that
\begin{multline*}\int_{\frac{1}{2}\cdot
B_{d+1}}\left(\int_{U_{\lambda_0}\left(x,t\,\big|\widetilde{X}(i;\,z,d+1)\right)}
\frac{\pds_{\left(\widetilde{X}\left(y,i;z,d+1\right)\right)_0}^2
\left(\widetilde{X}\left(y,i;z,d+1\right)\right)}
{\diam\left(\widetilde{X}\left(y,i;z,d+1\right)\right)^{d(d+1)}}\,\ud\mu(y)\right)
\di\mu(z)\leq\\\int_{U_{\lambda_0}\left(x,t\,\big|\widetilde{X}(i;d+1)\right)}
\frac{\pds_{\left(\widetilde{X}\left(y,i;z,d+1\right)\right)_0}^2
\left(\widetilde{X}\left(y,i;z,d+1\right)\right)}
{\diam\left(\widetilde{X}\left(y,i;z,d+1\right)\right)^{d(d+1)}}\,\ud\mu(y)\ud\mu(z)\leq\frac{\rho_2}{t^{d^2}}\cdot
c_{\mathrm{MT}}^2(x,t,\lambda_0).\end{multline*} Hence, by
Chebychev's inequality we obtain
$$\mu\left(\mathcal{Q}(\tau)\right)\geq\mu\left(\frac{1}{2}\cdot B_{d+1}\right)-\frac{t^d}{\tau}.$$
We thus fix
\begin{equation*}\label{equation:tauer}\tau_0=\tau_0(d,C_{\mu})>
\frac{2\cdot C_{\mu}}{\lambda_0^d},\end{equation*} and by the
$d$-regularity of $\mu$ we have the lower
bound\begin{equation}\label{equation:no-more-Chebychev}\mu(\mathcal{Q}(\tau_0))>
\frac{1}{2}\cdot\mu\left(\frac{1}{2}\cdot
B_{d+1}\right).\end{equation}Clearly,  one can  choose any
$\widetilde{x}_{d+1}$  in  $\mathcal{Q}(\tau_0)\not=\emptyset$ and
it will satisfy equation~(\ref{equation:hot-tin-roof}) with
$C_6=\tau_0\cdot\rho_2.$

Next, to choose $\tau$ such that $\mu(\mathcal{G}(\tau))$ is also
sufficiently large, i.e., to  find a point $\widetilde{x}_{d+1}$
which satisfies equation~(\ref{equation:dog-day-after}) as well, we
apply equation~(\ref{equation:krab}) and Chebychev's inequality to
obtain
$$\mu\left(\mathcal{G}(\tau)\right)\geq\mu\left(\frac{1}{2}\cdot B_{d+1}\right)-\frac{t^d}{\tau}.$$Hence,
for $\tau=\tau_0$, by the $d$-regularity of $\mu$ we get that
\begin{equation}\label{equation:lowly-line}\mu\left(\mathcal{G}(\tau_0)\right)>\frac{1}{2}
\cdot\mu\left(\frac{1}{2}\cdot B_{d+1}\right).\end{equation}

Finally, the combination of
equations~(\ref{equation:no-more-Chebychev})
and~(\ref{equation:lowly-line}) results in the
inequality$$\mu\left(\mathcal{Q}(\tau_0)\cap\mathcal{G}(\tau_0)\right)>0,$$
and therefore the lemma is established with $C_6$ (as specified
above)
and$$C_7=\tau_0\cdot\rho_1\cdot\frac{2^{(d+1)(d+2)}}{\omega_{\mu}}.$$\qed

\subsubsection{Deriving Proposition~\ref{proposition:sappy-homilies} from
Lemma~\ref{lemma:there-is-no-other}} \label{subsubsection:lookyling}
We arbitrarily fix an index $0\leq i\leq d$ and prove
Proposition~\ref{proposition:sappy-homilies} by specifying a
constant $C_5 =C_5(d,C_{\mu},\lambda_0)$ such that
\begin{equation}\label{equation:barry-manilow}\int_{B_i}\left(\frac{\DIst(y,L[\widetilde{X}(d+1)])}{t}\right)^2\,\ud\mu(y)
\leq C_5\cdot c_{\mathrm{MT}}^2(x,t,\lambda_0),\textup{ for all
}0\leq i\leq d\,.\end{equation}
Our strategy for proving equation~\eqref{equation:barry-manilow} is
to first show that for any point $\widetilde{x}_{d+1}$ satisfying
Lemma~\ref{lemma:there-is-no-other}, the following inequality holds
(for the fixed index $i$ and $C_6$ as in
Lemma~\ref{lemma:there-is-no-other}):
\begin{equation}\label{equation:mr-ed}\int_{B_i}\left(\frac{\DIst(y,L[\widetilde{X}(i\,;\widetilde{x}_{d+1},d+1)
])}{t}\right)^2\,\ud\mu(y)
\leq
\frac{2^{(d+1)(d+2)}}{\omega_{\mu}^2} \cdot \left(
\frac{2}{\lambda_0} \right)^{d(d+1)}\cdot C_6\cdot
c_{\mathrm{MT}}^2(x,t,\lambda_0).\end{equation}
Then, a
basic geometric argument shows that equation~\eqref{equation:mr-ed}
implies equation~\eqref{equation:barry-manilow}.

Now, if $\widetilde{x}_{d+1}$ is a point satisfying
Lemma~\ref{lemma:there-is-no-other}, then since
$\widetilde{x}_{d+1}\in B_{d+1}$ and the collection of balls
$\{B_i\}_{i=0}^{d+1}$ is $d$-separated we have that
\begin{equation}\label{equation:war-lord} B_i\subseteq U_{\lambda_0}(x,t\,\big|\widetilde{X}
(i;\widetilde{x}_{d+1},d+1))\end{equation} and
\begin{equation}\label{equation:holy-wawder}\MM_d(\widetilde{X}(i;\widetilde{x}_{d+1},d+1))\geq\omega_{\mu}\cdot
t^d.\end{equation} If $1\leq i\leq d,$ then
Lemma~\ref{lemma:hey-diddle-diddle} implies that for any $y\in B_i$
\begin{equation}\label{equation:diaper-dandy}\frac{\pds_{\widetilde{x}_0}^2
\left(\widetilde{X}(y,i\,;\widetilde{x}_{d+1},d+1)\right)}
{\diam\left(\widetilde{X}(y,i\,;\widetilde{x}_{d+1},d+1)\right)^{d(d+1)}}\geq
\frac{\omega_{\mu}^2}{2^{(d+1)(d+2)}}\cdot\left(\frac{\DIst(y,L[\widetilde{X}(i\,;\widetilde{x}_{d+1},d+1)
])}{t}\right)^2\cdot\frac{1}{t^{d(d+1)}}.\end{equation}
Combining this inequality with
equations~\eqref{equation:hot-tin-roof}
and~\eqref{equation:war-lord}, we conclude
equation~\eqref{equation:mr-ed} in this case.

However, if $i=0$, then we cannot directly apply
Lemma~\ref{lemma:hey-diddle-diddle}.  Instead, we first note that
for all $y\in B_0$
\begin{equation}\label{equation:star}\frac{\delta_{\mu}}{2} \cdot
t=\lambda_0\cdot
t\leq\min(\widetilde{X}(y,0\,;\widetilde{x}_{d+1},d+1))
\leq\diam(\widetilde{X}(y,0\, ;\widetilde{x}_{d+1},d+1))\leq 2\cdot
t.\end{equation}
Then, using these bounds we apply the law of sines
for the polar sine (see equation~(\ref{equation:law-of-sines})) to
obtain the lower bound
\begin{multline}\label{equation:swap-sines}\pds_{y}\left(\widetilde{X}(y,0\,;\widetilde{x}_{d+1},d+1)\right)
\geq\\\left(\frac{\min\left(\widetilde{X}(y,0\,;\widetilde{x}_{d+1},d+1)\right)}
{\diam\left(\widetilde{X}(y,0\,;\widetilde{x}_{d+1},d+1)\right)}\right)^{\frac{d(d+1)}{2}}
\cdot\pds_{\widetilde{x}_1}\left(\widetilde{X}(y,0\,;\widetilde{x}_{d+1},d+1)\right)
\geq\\ \left(\frac{\lambda_0}{2}\right)^{\frac{d(d+1)}{2}}
\cdot\pds_{\widetilde{x}_1}\left(\widetilde{X}(y,0\,;\widetilde{x}_{d+1},d+1)\right).\end{multline}
Applying Lemma~\ref{lemma:hey-diddle-diddle} to the RHS of
equation~(\ref{equation:swap-sines}), and then applying the RHS
inequality of equation~(\ref{equation:star}) to the resulting
equation gives the inequality
\begin{multline}\label{equation:if-index-is-0}\frac{\pds_{y}^2
\left(\widetilde{X}(y,0\,;\widetilde{x}_{d+1},d+1)\right)}
{\diam\left(\widetilde{X}(y,0\,;\widetilde{x}_{d+1},d+1)\right)^{d(d+1)}}\geq\\
\frac{\omega_{\mu}^2}{2^{(d+1)(d+2)}} \cdot \left(
\frac{\lambda_0}{2}\right)^{d(d+1)} \cdot \left(\frac{\DIst(y,L
[\widetilde{X}(0\,;\widetilde{x}_{d+1},d+1)
])}{t}\right)^2\cdot\frac{1}{t^{d(d+1)}}.\end{multline}
We replace equation~(\ref{equation:diaper-dandy}) (where $1\leq
i\leq d$) and equation~(\ref{equation:if-index-is-0}) (where $i=0$)
by the following equation which holds for all  $0\leq i\leq d$
\begin{multline}\label{equation:alley-gator}\frac{\pds_{\left(\widetilde{X}(y,i\,;\widetilde{x}_{d+1},d+1)\right)_0}^2\left(\widetilde{X}(y,i\,;\widetilde{x}_{d+1},d+1)\right)}
{\diam\left(\widetilde{X}(y,i\,;\widetilde{x}_{d+1},d+1)\right)^{d(d+1)}}\geq\\
\frac{\omega_{\mu}^2}{2^{(d+1)(d+2)}} \cdot
\left(\frac{\lambda_0}{2}\right)^{d(d+1)}
\cdot\left(\frac{\DIst(y,L[\widetilde{X}(i\,;\widetilde{x}_{d+1},d+1)
])}{t}\right)^2\cdot\frac{1}{t^{d(d+1)}}.\end{multline} Combining
equation~(\ref{equation:alley-gator}) with
equations~(\ref{equation:hot-tin-roof})
and~(\ref{equation:war-lord}) implies
equation~\eqref{equation:mr-ed} for the fixed index $i$.

Next, equation~(\ref{equation:mr-ed}) implies
equation~(\ref{equation:barry-manilow})  via the following argument.
We first note that the elements $\widetilde{X}(d+1)$,
$\widetilde{X}(i;\widetilde{x}_{d+1},d+1)$, and
$\widetilde{X}(i;d+1)$ are all non-degenerate, and we define the
orthogonal projections
\begin{alignat*}{3}
P_{d+1} \ & :&& \ \HH \rightarrow L[\widetilde{X}(d+1)]\,,\\
P_i \ & :&& \ \HH \rightarrow L[\widetilde{X}(i;\widetilde{x}_{d+1},d+1)]\,,\\
P_{i,d+1} \ & :&& \ \HH\rightarrow L[\widetilde{X}(i;d+1)]\,.
\end{alignat*}Using these projections we can reduce the situation to  two cases.

The first  is when $P_{d+1}=P_i$, that is,
$L[\widetilde{X}(d+1)]=L[\widetilde{X}(i\,;\widetilde{x}_{d+1},d+1)].$
In this case equation~(\ref{equation:barry-manilow}) holds trivially
by equation~(\ref{equation:mr-ed}) with
$$C_5\geq
\frac{2^{(d+1)(d+2)}}{\omega_{\mu}^2} \cdot \left(
\frac{2}{\lambda_0} \right)^{d(d+1)} \cdot C_6\,.$$

The second  is when $P_{d+1}\not=P_i$. In this case, we rely on the
following inequality.
\begin{multline}\label{equation:hell-froxe-over}\left(\frac{\DIst\left(y,P_{d+1}(y)\right)}{t}\right)^2\leq\left(\frac{\dist(y,P_{d+1}\left(P_i(y)\right)}
{t}\right)^2\leq\\2\cdot\left[\left(\frac{\DIst\left(y,P_i(y)\right)}{t}\right)^2
+\left(\frac{\DIst\left(P_i(y),P_{d+1}(P_i(y))\right)}{t}\right)^2\right].\end{multline}
Integrating the inequality of
equation~(\ref{equation:hell-froxe-over}) over the ball $B_i$ and
applying the inequality of equation~(\ref{equation:mr-ed}), we
obtain the bound
\begin{multline}\label{equation:green-acres}\int_{B_i}\left(\frac{\DIst\left(y,P_{d+1}(y)\right)}
{t}\right)^2\,\ud\mu(y)\leq\\
2\cdot\frac{2^{(d+1)(d+2)}}{\omega_{\mu}^2} \cdot \left(
\frac{2}{\lambda_0} \right)^{d(d+1)}\cdot C_6\cdot
c_{\mathrm{MT}}^2(x,t,\lambda_0)+2\,\int_{B_i}\left(\frac{\DIst\left(P_i(y),P_{d+1}(P_i(y))\right)}{t}\right)^2\,\ud\mu(y).\end{multline}
The only thing remaining is to bound the second term on the RHS of
equation~(\ref{equation:green-acres}).

Since $P_{d+1}\not= P_i$, the $d$-planes $L[\widetilde{X}(d+1)]$ and
$L[\widetilde{X}(i\,;\widetilde{x}_{d+1},d+1)]$ are distinct. Let
$\alpha$ denote the dihedral angle between these two $d$-planes
along their intersection, the  $(d-1)$-plane
$L[\widetilde{X}(i;d+1)].$  We note that $\sin(\alpha)>0$.
Furthermore, for all $y\in B(x,t)$ we have that
\begin{equation}\label{equation:slappy-nappy}\dist\left(P_i(y),P_{d+1}(P_i(y))\right)=\sin(\alpha)\cdot
\dist\left(P_i(y),P_{i,d+1}(P_i(y))\right).\end{equation}

We can bound the RHS of equation~(\ref{equation:slappy-nappy}) by
bounding each of the factors separately. For any $j\not=i, d+1$, we
have the inclusion
$$\widetilde{x}_j=(\widetilde{X}(d+1))_j\in B(x,t)\cap
L[\widetilde{X}(i;d+1)]\subseteq
L[\widetilde{X}(i;\widetilde{x}_{d+1},d+1)].$$Hence we obtain that
for all $y\in B(x,t)$
\begin{equation}\label{equation:lincoln-getty}\dist(P_i(y),P_{i,d+1}(P_i(y)))\leq\|P_i(y)-
\widetilde{x}_j\|=\|P_i(y-\widetilde{x}_j)\|
\leq\|y-\widetilde{x}_j\|\leq2\cdot t\,.\end{equation} To bound
$\sin(\alpha),$ we observe that
\begin{equation}\label{equation:eddie-money}\sin(\alpha)\,=\,\frac{\DIst\left(\widetilde{x}_{d+1},P_{d+1}(\widetilde{x}_{d+1})
\right)}{\DIst\left(\widetilde{x}_{d+1},P_{i,d+1}(\widetilde{x}_{d+1})\right)}.\end{equation}
By equation~(\ref{equation:holy-wawder}) and the product formula for
contents we have
that$$\dist\left(\widetilde{x}_{d+1},P_{i,d+1}(\widetilde{x}_{d+1})\right)\geq\frac{\omega_{\mu}}{2^{d-1}}\cdot
t.$$Combining this lower bound with
equation~(\ref{equation:eddie-money}) we have the upper bound
\begin{equation}\label{equation:sine-thing-dihedral}\sin(\alpha)\leq\,\frac{2^{d-1}}{\omega_{\mu}}\cdot
\frac{\DIst\left(\widetilde{x}_{d+1},P_{d+1}(\widetilde{x}_{d+1})
\right)}{t}=\frac{2^{d-1}}{\omega_{\mu}}\cdot
\frac{\DIst\left(\widetilde{x}_{d+1},L[\widetilde{X}(d+1)])
\right)}{t}\,.\end{equation}
Applying equations~(\ref{equation:lincoln-getty})
and~(\ref{equation:sine-thing-dihedral}) to the RHS of
equation~(\ref{equation:slappy-nappy}), we have the following
uniform upper bound for all $y\in B(x,t)$
\begin{equation}\label{equation:frankenwine}\DIst\left(P_i(y),P_{d+1}(P_i(y))\right)\leq\frac{2^d}{\omega_{\mu}}\cdot
\DIst\left(\widetilde{x}_{d+1},L[\widetilde{X}(d+1)]\right).\end{equation}

Equation~(\ref{equation:frankenwine}),
Lemma~\ref{lemma:there-is-no-other} and the $d$-regularity of $\mu$
imply that
\begin{multline}\label{equation:sandy-duncan}\int_{B_i}\left(\frac{\DIst\left(P_i(y),P_{d+1}(P_i(y))\right)}
{t}\right)^2\,\ud\mu(y)\leq\frac{4^d}{\omega_{\mu}^2}
\cdot\left(\frac{\DIst\left(\widetilde{x}_{d+1},L[\widetilde{X}(d+1)]\right)}{t}\right)^2
\mu(B_i)\leq\\\frac{4^d\cdot C_{\mu}\cdot C_7}{\omega_{\mu}^2}\cdot
c_{\mathrm{MT}}^2(x,t,\lambda_0).\end{multline}
Finally, applying equation~(\ref{equation:sandy-duncan}) to the RHS
of equation~(\ref{equation:green-acres}) finishes the proof of
equation~(\ref{equation:barry-manilow}), and thus concludes
Proposition~\ref{proposition:sappy-homilies}.\qed

\section{The Proof of Theorem~\ref{theorem:funky-salmon}}\label{section:fubani}
Theorem~\ref{theorem:funky-salmon} is an easy consequence of
Theorem~\ref{theorem:doozy} and the following proposition, which
actually holds for any non-negative function on $H^{d+2}$ instead of
$c_{\mathrm{MT}}$.
\begin{proposition}\label{proposition:little-fubini}If
$\lambda>0$, then
\begin{equation}\label{equation:hell-no-we}\int_{B}\,\int_0^{\diam(B)}c_{\mathrm{MT}}^2(x,t,\lambda)\frac{\ud
t}{t^{d+1}}\fdi\mu(x)\leq \left(\frac{2}{\lambda}\right)^d\cdot
\frac{C_{\mu}}{d}\cdot c_{\mathrm{MT}}^2(\mu|_{3\cdot
B},\lambda/2),\end{equation}for any ball $B\subseteq
H$.\end{proposition}Indeed, Theorem~\ref{theorem:doozy} and the
$d$-regularity of $\mu$ imply that
$$\beta_2^2(x,t)\lessapprox\frac{c_{\mathrm{MT}}^2(x,t,\lambda_0)}{t^d},
\textup{ for all }x\in\Supp \text{ and } 0<t\leq\diam(\Supp),$$where
the comparison depends only on $d$ and $C_{\mu}$.  Thus, by
Proposition~\ref{proposition:little-fubini} we obtain the following
estimate for all balls $B$ such that $\diam(B)\leq\diam(\Supp)$:
$$J_2(\mu|_B)\,=\,\int_{B}\,\int_0^{\diam(B)}\beta_2^2(x,t)\frac{\ud
t}{t}\fdi\mu(x)\ \lessapprox\  c_{\mathrm{MT}}^2(\mu|_{3\cdot
B},\lambda_0/2)\,\leq\ c_{\mathrm{MT}}^2\left(\mu|_{3\cdot
B}\right),$$where again the comparison depends only on $d$ and
$C_{\mu}$.

\qed

The rest of this section proves
Proposition~\ref{proposition:little-fubini}.
\subsection{Proof of Proposition~\ref{proposition:little-fubini}}

\subsubsection{Preliminary Notation and Observations}For any $x\in H$ and $0< t<\infty$
we note the following trivial inclusion
\begin{equation}\label{equation:sufficiently-in-separated}U_{\lambda}(B(x,t)) \subseteq W_{\lambda/2}(B(x,t)).\end{equation}
If $B$ is a ball of finite diameter in $H$, let
\begin{equation*}\label{equation:super-saturday}\WW(B)=\left\{(x,X,t)\in\, B\times H^{d+2}\times
\left(0,\diam(B)\right]:\,X\in
U_{\lambda}(B(x,t))=\right\}.\end{equation*} For fixed $(x,X)\in
B\times H^{d+2}$, we define the slice of $\WW(B)$ corresponding to
$(x,X)$:\begin{equation*}\label{equation:itty-bitty}\WW\left(B\big|\,x,X\right)=\left\{\,t>0:(x,X,t)\in
\WW(B)\right\},\end{equation*}and we note that %if
%$\WW\left(B\big|\,x,X\right)$ is non-empty, then
%$$\max_{0\leq i\leq d+1}\|x_i-x\|\leq t\leq \frac{\min(X)}{\lambda},\textup{ for all }t\in\WW\left(B\big|\,x,X\right)$$and
%moreover
\begin{equation}\label{equation:lefty-loo}\WW\left(B\big|\,x,X\right)=\left[\max_{0\leq
i\leq
d+1}\|x_i-x\|\,,\,\frac{\MIn(X)}{\lambda}\right]:=\left[u_1(x,X),u_2(x,X)\right].\end{equation}

We define the following two projections.  Let $P_{1,2}:H\times
H^{d+2}\times (0,\infty)\rightarrow H\times H^{d+2}$ be  such that
$P_{1,2}(x,X,t)=(x,X)$, and let $P_2:H\times H^{d+2}\rightarrow
H^{d+2}$ be the projection such that $P_2(x,X)=X.$ We also adopt the
harmless convention of taking $P_2(x,X,t)=P_2(x,X)=X$.

At last we note that the combination of
equation~(\ref{equation:sufficiently-in-separated}) and the
definition of $\WW(B)$ implies the inclusion
\begin{equation}\label{equation:projection-sufficiently}P_2\left(\WW(B)\right)\subseteq
W_{\lambda/2}(3\cdot B).\end{equation}

\subsubsection{Details of the Proof } We first apply Fubini's
Theorem and the definition of $c^2_{\mathrm{MT}}(x,t,\lambda)$ to
rewrite the integral on the LHS of
equation~(\ref{equation:hell-no-we}) in the following form
\begin{equation}\label{equation:dundee}\int_{B}\,\int_0^{\diam(B)}c_{\mathrm{MT}}^2(x,t,\lambda)\frac{\ud
t}{t^{d+1}}\ud\mu(x)\,=\,\int_{P_{1,2}\left(\WW(B)\right)}c_{\mathrm{MT}}^2(X)\left(\int_{\WW\left(B|x,X\right)}\frac{\ud
t}{t^{d+1}}\right)\ud\mu^{d+3}(x,X).\end{equation} Then, for
$(x,X)\in P_{1,2}\left(\WW(B)\right)$ such that
$\WW\left(B\big|\,x,X\right)\not=\emptyset $, let
\begin{equation}\label{equation:add-a-number}I(x,X)=\int_{\WW\left(B|x,X\right)}\frac{\ud
t}{t^{d+1}}.\end{equation}In view of
equation~(\ref{equation:lefty-loo}) we get that
\begin{equation*}I(x,X)=\int_{u_1(x,X)}^{u_2(x,X)}\frac{\ud
t}{t^{d+1}}.\end{equation*}
We note that $u_1(x,X)>0$ a.e.~on
$P_{1,2}\left(\WW(B)\right)$ (w.r.t.~$\mu^{d+3}$), and thus
$I(x,X)<\infty$ a.e.~on $P_{1,2}\left(\WW(B)\right)$. This yields
the inequality
\begin{equation}\label{equation:clint}I(x,X)=\frac{1}{d}\cdot\left(\frac{1}{u_1(x,X)^d}-
\frac{1}{u_2(x,X)^d}\right)\leq\frac{1}{d}\cdot\frac{1}
{\displaystyle\max_{0\leq i\leq d+1}\|x_i-x\|^d}\  \textup{ a.e.~on
}P_{1,2}\left(\WW(B)\right).\end{equation}
Moreover, we can restrict
our attention to $(x,X)$ such that  $I(x,X)>0$. Defining the set
\begin{equation}\label{equation:lonely-lolita}I^{-1}(0,\infty)=\{(x,X)\in
P_{1,2}\left(\WW(B)\right):0<I(x,X)<\infty\},\end{equation}and
combining equations~(\ref{equation:dundee})-(\ref{equation:clint})
we obtain the inequality
\begin{equation}\label{equation:grunty}\int_{B}\,\int_0^{\diam(B)}c_{\mathrm{MT}}^2(x,t,\lambda)\frac{\ud
t}{t^{d+1}}\ud\mu(x)\leq\frac{1}{d}
\int_{I^{-1}(0,\infty)}\frac{c_{\mathrm{MT}}^2(X)}{\displaystyle\max_{0\leq
i\leq d+1}\|x_i-x\|^d}\,\ud\mu^{d+3}(x,X).\end{equation}
In order to estimate the RHS of equation~(\ref{equation:grunty}) we
again apply Fubini's Theorem. More specifically, for any $X\in
P_2\left(I^{-1}(0,\infty)\right)$ we define
$$ I^{-1}(0,\infty)\big|X=\{x\in H:(x,X)\in I^{-1}(0,\infty)\},$$
and thus rewrite equation~(\ref{equation:grunty}) as follows.
\begin{multline}\label{equation:grunty2}
\int_{B}\,\int_0^{\diam(B)}c_{\mathrm{MT}}^2(x,t)\frac{\ud
t}{t^{d+1}}\ud\mu(x)
\leq\\
\frac{1}{d}\,\int_{P_2\left(I^{-1}(0,\infty)\right)}c_{\mathrm{MT}}^2(X)\left[\int_{I^{-1}(0,\infty)\big|
X}\frac{\ud\mu(x)}{\displaystyle\max_{0\leq i\leq
d+1}\|x_i-x\|^d}\right]\ud\mu^{d+2}(X) \,.\end{multline}

We next bound the inner integral of the above equation for
a.e.~$X\in P_2\left(I^{-1}(0,\infty)\right)$, that is, the integral
\begin{equation}\label{equation:korea}\int_{ I^{-1}(0,\infty)\big|X}\frac{\ud\mu(x)}{\displaystyle\max_{0\leq
i\leq d+1}\|x_i-x\|^d}\,.\end{equation}
If $X\in P_2\left(I^{-1}(0,\infty)\right)$ and $x\in
I^{-1}(0,\infty)\big| X$, then
\begin{equation*}\label{equation:mulberry-street}\frac{\MIn(X)}{\lambda}>\max_{0\leq i\leq
d+1}\|x_i-x\|>0.\end{equation*}Hence, for fixed $X\in
 P_2\left(I^{-1}(0,\infty)\right)$, with $x_0=(X)_0$, we have the set inclusion
$$I^{-1}(0,\infty)\big| X\subseteq
B\left(x_0,\lambda^{-1}\cdot\MIn(X)\right).$$Thus the integral of
equation~(\ref{equation:korea}) is bounded by
\begin{equation*}\label{equation:tonkin}\int_{B\left(x_0,\lambda^{-1}\cdot\MIn(X)\right)}\frac{\ud\mu(x)}
{\displaystyle\max_{0\leq i\leq d+1}\|x_i-x\|^d}.\end{equation*}
Furthermore, $\MIn(X)>0$ a.e.~on $H^{d+2},$ and by the triangle
inequality we obtain
$$\MIn(X)\leq2\cdot\max_{0\leq i\leq d+1}\|x_i-x\|.$$Combining these observations with the upper bound of
equation~(\ref{equation:upper-regularity}), we have the following
inequality for a.e.~$X\in P_2\left(I^{-1}(0,\infty)\right)$:
\begin{equation*}\int_{I^{-1}(0,\infty)\big| X}\frac{\ud\mu(x)}{\displaystyle\max_{0\leq i\leq
d+1}\|x_i-x\|^d}\leq
2^d\int_{B\left(x_0,\lambda^{-1}\cdot\MIn(X)\right)}\frac{\ud\mu(x)}{\MIn(X)^d}\leq
\left(\frac{2}{\lambda}\right)^d\cdot C_{\mu}.\end{equation*}

Applying this uniform bound to the RHS of
equation~(\ref{equation:grunty2}) we get that
\begin{equation}\label{equation:grunty-1}
\int_{B}\,\int_0^{\diam(B)}c_{\mathrm{MT}}^2(x,t)\frac{\ud
t}{t^{d+1}}\ud\mu(x) \leq
\left(\frac{2}{\lambda}\right)^d\cdot\frac{C_{\mu}}
{d}\,\int_{P_2\left(I^{-1}(0,\infty)\right)}c_{\mathrm{MT}}^2(X)\,\ud\mu^{d+2}(X).\end{equation}
Further application of
equations~(\ref{equation:projection-sufficiently})
and~\eqref{equation:lonely-lolita} bounds the integral on the RHS of
equation~(\ref{equation:grunty-1}) by
$c_{\mathrm{MT}}^2(\mu|_{3\cdot B},\lambda/2)$ and thus concludes
the proof.\qed

\section{A Menagerie of Curvatures}\label{section:curvatures}
Here we discuss a variety of curvatures for $d$-regular measures on
$H$. In Subsection~\ref{subsection:squared-curvatures} we describe
some curvatures that can be used to characterize uniform
rectifiability, while indicating two levels of information needed
for this purpose. In Subsection~\ref{subsection:p-q} we briefly give
an example of continuous curvatures that can be used to quantify the
$(p,p)$-geometric property ($1\leq p<\infty$) of David and
Semmes~\cite{DS93}. Finally, in Subsection~\ref{subsection:leger} we
discuss our doubts about the utility of a previously suggested
curvature for the purposes of implying the rectifiability of
$\mu$~\cite[Theorem~0.3]{Leger}.
\subsection{Curvatures Characterizing Uniform Rectifiability}\label{subsection:squared-curvatures}
%
%There are a variety of discrete curvatures that can be used to
%investigate uniform rectifiability.
%
We start with a few continuous curvatures that are completely
equivalent to the Jones-type flatness (in the sense of
Theorems~\ref{theorem:upper-main} and~\ref{theorem:funky-salmon}).
They thus characterize the uniform rectifiability of $\mu$ by the
criterion that the ratios between the curvatures of $\mu|_B$ and the
corresponding measures $\mu(B)$ are uniformly bounded for all balls
$B$ in $H$. We remark that here the continuous curvature of $\mu|_B$
is obtained by integrating a corresponding discrete curvature over
all $(d+1)$-simplices in $B^{d+2}$.

It is also possible to use a coarser level of information by
introducing the parameter $\lambda$ and modifying the continuous
curvature of $\mu|_B$ by integrating over the well-scaled set of
simplices $W_{\lambda}(B)$ (see equation~\eqref{equation:big-W}). In
the case of the Menger-type curvatures, both types of continuous
curvatures are comparable (up to possible blow ups of the ball $B$).
We thus say that the Menger-type curvature is stable (when $\lambda$
approaches zero). In Subsection~\ref{subsubsection:tilde-curvature}
we present a discrete curvature for which we can easily compare the
Jones-type flatness with the latter type of continuous curvature
(with parameter $\lambda$). Currently, we cannot decide if this
curvature is stable and thus cannot use the former version of
continuous curvature to characterize uniform rectifiability.

\subsubsection{Additional Stable Curvatures}\label{subsec:menger_type_zoo}We define
the following discrete curvatures
$$c_{\mathrm{min}}(X)=\sqrt{\frac{\displaystyle\min_{0\leq i\leq d+1}\pds_{x_i}^2(X)}{\diam(X)^{d(d+1)}}}\,,$$
$$c_{\mathrm{vol}}(X)=\sqrt{\frac{\MM^{2}_{d+1}(X)}{\diam(X)^{(d+1)(d+2)}}}\,,$$and$$c_{\mathrm{max}}(X)
=\sqrt{\frac{\max_{0\leq i\leq
d+1}\pds_{x_i}^2(X)}{\diam(X)^{d(d+1)}}}\,.$$
We note that for all
$X\in H^{d+2}$
with
$\min(X)>0$:\begin{equation*}\label{equation:stinky-pinky}c_{\mathrm{MT}}^2(X)\approx
c^2_{\mathrm{max}}(X)\geq c_{\mathrm{min}}^2(X)\geq
c_{\mathrm{vol}}^2(X)\,.\end{equation*}
Furthermore, we note that
for $\lambda>0$ and any $1$-separated element $X$,
$$c_{\mathrm{vol}}^2(X)\geq\lambda^{2(d+1)}\cdot c_{\mathrm{MT}}^2(X).$$

As such, analogous estimates to those of
Theorems~\ref{theorem:upper-main} and~\ref{theorem:funky-salmon}
hold for the curvatures $c_{\mathrm{min}}$, $c_{\mathrm{vol}}$, and
$c_{\mathrm{max}}$. In particular, the continuous curvatures
(integrated over all simplices in corresponding products of balls)
$c_{\mathrm{min}}$, $c_{\mathrm{MT}}$, $c_{\mathrm{vol}}$ and
$c_{\mathrm{max}}$ are comparable (up to possible blow ups of the
underlying balls).

We can further extend this collection of stable curvatures. For
example, we may include any order statistics of the p-sines of
vertices of a given simplex (replacing the maximum or the minimum,
which are used in $c_{\mathrm{max}}$ and $c_{\mathrm{min}}$
respectively).

\subsubsection{An Algebraic Curvature with Questionable
Stability}\label{subsubsection:tilde-curvature}For $X\in H^{d+2}$,
let
\begin{equation*} c_{\mathrm{alg}}(X)=\begin{cases}\displaystyle\frac{\pds_{x_0}(X)}{\displaystyle\prod_{1\leq i<j\leq d+1}\|x_i-x_j\|},&\textup{ if }\min(X)>0,\\
\quad\quad\quad\quad 0,&\textup{ otherwise.
}\end{cases}\end{equation*}
We see that unlike the curvatures of
Subsection~\ref{subsec:menger_type_zoo}, this one is algebraic. In
fact, it is the invariant ratio of the law of polar sines expressed
in equation~\eqref{equation:law-of-sines}.

We trivially have the inequality
\begin{equation*}\label{equation:burrito} c_{\mathrm{alg}}^2(X)\geq c_{\mathrm{MT}}^2(X),\textup{ for all }X\in
H^{d+2}.\end{equation*}
Furthermore,  if $X$ is $1$-separated for
$\lambda>0$, then we also have the opposite inequality:
$$c_{\mathrm{MT}}^2(X)\geq\lambda^{d(d+1)}\cdot c_{\mathrm{alg}}^2(X).$$
Hence, for any $0<\lambda\leq\lambda_0$ (where $\lambda_0$ is the
constant suggested by Theorem~\ref{theorem:doozy}) and all balls
$B\subseteq H$ we have the estimate (with constants depending on
$d$, $C_{\mu}$, and $\lambda$):
\begin{equation} \label{eq:compare_with_lambda} J_2(\mu|_{\frac{1}{3}\cdot B})\lessapprox
 c_{\mathrm{alg}}^2(\mu|_B,\lambda/2)\lessapprox J_2(\mu|_{6\cdot
B})\,.\end{equation}
Thus, fixing such $\lambda$ one can use the curvatures
$\{c_{\mathrm{alg}}^2(\mu|_B,\lambda/2)\}_{B \subseteq H}$ to
characterize uniform rectifiability.

We are not sure whether one can replace the curvature
$c_{\mathrm{alg}}^2(\mu|_B,\lambda/2)$ in
equation~\eqref{eq:compare_with_lambda} by the curvature
$c_{\mathrm{alg}}^2(\mu|_B)$ (even with the introduction of blow ups
of the ball $B$ in that equation). That is, we cannot decide at this
point if the algebraic curvature is stable or not. It is clear
though that our methods for controlling $c_{\mathrm{MT}}^2(\mu|_B)$
by $J_2(\mu|_{6\cdot B})$ (as expressed in
Theorem~\ref{theorem:upper-main} and established in~\cite{LW-part1})
are insufficient for controlling $c_{\mathrm{alg}}^2(\mu|_B)$ by
$J_2(\mu|_{C \cdot B})$ for some $C>1$ (independent of $B$).
%We also note that the discrete curvature $
%c_{\mathrm{alg}}(X)$ is clearly more singular than
%$c_{\mathrm{MT}}(X)$.

\subsection{Curvatures Characterizing the $(p,p)$-Geometric Property ($1\leq p<\infty$)}\label{subsection:p-q}For
$1\leq p<\infty$, let
\begin{equation*} \widetilde{J}_p(\mu|_B)=\int_{B}\int_{0}^{\diam(B)}\beta_p^p(x,t)\di\mu(x)\frac{\di
t}{t}.\end{equation*}
We note that if $p \neq 2$, then $\widetilde{J}_p(\mu|_B)$ differs
from ${J}_p(\mu|_B)$ (defined in
equation~\eqref{equation:jone-flat-int}) in the power of
$\beta_p(x,t)$.
 A $d$-regular measure $\mu$ on $H$ satisfies the
$(p,p)$-geometric property~\cite[Part~IV]{DS93} if there exists a
constant $C=C(\mu)$ such that
\begin{equation*}\label{equation:numbly-number} \widetilde{J}_p(\mu|_B)\leq C\cdot\mu(B),\textup{ for all
balls }B\subseteq H.\end{equation*}

The methods of this paper and~\cite{LW-part1} extend to comparing
$\widetilde{J}_p$ with a different kind of continuous curvature as
follows.
%For $1\leq p<\infty$, let
%$$c_{(p,p)}(X)=\sqrt[p]{\frac{\pds_{x_0}(X)}{\diam(X)^{d(d+1)}}}$$
% you have a mistake with your scaling!
\begin{theorem}\label{theorem:higher-beta-integral}If $\mu$ is a
$d$-regular measure on $H$ and $1\leq p<\infty$, then there exists a
constant $C_8=C_8(d,C_{\mu},p)$ such that
\begin{equation*}\label{equation:higher-beta}\frac{1}{C_8}\cdot  \widetilde{J}_p\left(\mu|_{\frac{1}{3}\cdot B}\right)\leq
\int_{B^{d+2}}\frac{\pds_{x_0}^p(X)}{\diam(X)^{d(d+1)}}
\di\mu^{d+2}(X)\leq\\C_8\cdot
 \widetilde{J}_p\left(\mu|_{6\cdot B}\right),\end{equation*}for
all balls $B\in H$\end{theorem} We thus obtain the following
characterization of the $(p,p)$-geometric property.
\begin{corollary} \label{cor:pp} If $1 \leq p < \infty$ and $\mu$ is a $d$-regular measure on $H$, then $\mu$ satisfies the
$(p,p)$-geometric property if and only if there exists a constant
$C=C(d,C_{\mu})$ such that
$$\int_{B^{d+2}}  \frac{\pds_{x_0}^p(X)}{\diam(X)^{d(d+1)}} \di\mu^{d+2}(X)\leq C\cdot\mu(B)$$
for all balls $B\subseteq H$.\end{corollary}
If $p=2$ then Theorem~\ref{theorem:higher-beta-integral} coincides
with the combination of Theorems~\ref{theorem:upper-main}
and~\ref{theorem:funky-salmon}. Similarly, in that case
Corollary~\ref{cor:pp} coincides with~\cite[Theorem~1.3]{LW-part1}.

\subsection{A Previously Suggested Curvature}
\label{subsection:leger}

L{\'e}ger~\cite{Leger} proposed the following discrete curvature for
$(d+1)$-simplices $X=(x_0,\ldots,x_{d+1})\in H^{d+2}$ where
$d\geq1$:
\begin{equation*}\label{equation:def-leger-curv}
c^{d+1}_L(X)=\frac{\dist(x_0,L[X(0)])^{d+1}}
{\prod_{i=1}^{d+1}\|x_i-x_0\|^{d+1}},\end{equation*}and the
corresponding continuous curvature for $\mu$ restricted to any ball
$B\subseteq H$
$$c_L^{d+1}(\mu|_B)=\int_{B^{d+2}}c_L^{d+1}(X)\di\mu^{d+2}(X).$$
If $d=1$, then his curvature coincides with the Menger
curvature~\cite{MMV96,M30} (up to multiplication by a constant). He
showed how to use his curvature in that case to infer rectifiability
properties of $\mu$. In particular, he established
Theorem~\ref{theorem:funky-salmon} when $d=1$.

L{\'e}ger's approach for proving the same type of results for
$d\geq2$ (while using the curvature $c_L^{d+1}(X)$) ostensibly
requires a bound of the form
\begin{equation*}\label{equation:wanted-no}J_2(\mu|_B)\lessapprox
c_L^{d+1}(\mu|_{3\cdot B}),\end{equation*}thus
generalizing~\cite[Lemma~2.5]{Leger}. However, any analysis or
adaptation of the proof of~\cite[Lemma~2.5]{Leger} to the curvature
$c_L^{d+1}(X)$, where $d>1$, seems to give at best the following
lower bound.
\begin{proposition}\label{proposition:leger-true}There exists
a constant $C_9=C_9(d,C_{\mu})$ such that
\begin{equation*} \widetilde{J}_{d+1}\left(\mu|_{\frac{1}{3}\cdot B}\right)\leq C_9 \cdot c_L^{d+1}(\mu|_B)\end{equation*}for any ball
$B\subseteq H$.
\end{proposition}
If $d>1$, then the function $\widetilde{J}_{d+1}$ (discussed in the
previous subsection) can be significantly smaller than the required
quantity $J_2$ (especially for very large $d$). This function also
characterizes the $(d+1,d+1)$-geometric property (see
Corollary~\ref{cor:pp}), which includes measures that are not
uniformly rectifiable whenever $d>1$.

\section{Open and New Directions}\label{section:conclusion}

We conclude the work presented here as well as in~\cite{LW-part1} by
suggesting possible directions for extending it. Most of them are
wide open.

\subsubsection*{$L_2(\mu)$ boundedness of $d$-dimensional Riesz transform}
Mattila, Melnikov and Verdera~\cite{MMV96} used the one-dimensional
Menger curvature to show that a one-regular measure $\mu$ is
uniformly rectifiable if and only if the one-dimensional Riesz
transform is bounded in $L_2(\mu)$ (they actually showed it for the
Cauchy kernel where $H=\cmplx$, but their analysis extends easily to
the former case). Farag~\cite{Farag99} showed that direct
generalization of their arguments to higher-dimensional Riesz
transforms are not possible. Nevertheless, one might suggest
alternative strategies to study whether $L_2(\mu)$ boundedness of
the Riesz kernel implies that for all balls $B \subseteq H$ the
quantity $c_{\mathrm{MT}}(\mu|_B)/\mu(B)$ is finite. In~\cite{GL}
numerical experiments have been performed in order to test some
heuristic strategies for such study. However, at this stage they
have not advanced our theoretical understanding of the problem.

\subsubsection*{Calculus of curvatures}
Our work suggests various computational techniques for obtaining
careful estimates of the high-dimensional Menger-type curvatures.
Our writing indicates the lack of some very basic machinery for
those curvatures. Indeed, while our analysis was based on very
elementary ideas, its writing required substantial development.
Nevertheless, the techniques established here could be taken for
granted in subsequent papers. Still more technical tools need to be
developed, and in particular we are interested in tools leading to
solutions of the following straightforward questions: 1.
determination of stability of the algebraic curvature
$c_{\mathrm{alg}}$ when $d>1$ (see
Subsection~\ref{subsubsection:tilde-curvature}); 2. characterization
of the $(p,q)$ geometric property~\cite[Part~IV]{DS93} by the
Menger-type curvature for special cases where $p \neq q$; 3.
relating the Menger-type curvature with the functions $J_p$ for all
$1 \leq p < 2\cdot d/(d-2)$ (we also wonder if such a relation is
different for the two cases: $p<2$ and $p>2$).
%4.
%developing techniques for estimating the Menger-type curvatures of
%measures which are not $d$-regular (so that one cannot invoke the
%Jones-type flatness).

\subsubsection*{Rectifiability by Menger-type curvatures}
When $d=1$ L{\'e}ger~\cite{Leger} formulated a modified version of
Theorem~\ref{theorem:funky-salmon} for more general measures and
used it to establish a criterion for rectifiability for another
large class of measures. His analysis immediately extends to our
setting, in particular, it implies that finiteness of the
$d$-dimensional Menger-type curvatures of certain measures is a
sufficient condition for $d$-dimensional rectifiability. We ask
about a necessary condition for rectifiability (as sharp as
possible) formulated in terms of the Menger-type curvatures.

\subsubsection*{Extensions to noisy setting}
Ahlfors regular measures (i.e., $d$-regular) are rather synthetic
for real applications. In~\cite{LW-volume} we extend both
Proposition~\ref{proposition:upper-bound-big-scale} and
Theorem~\ref{theorem:doozy} to a wide class of probability
distributions, in particular, distributions with ``additive noise''
around $d$-dimensional surfaces. An application of this result
appears in~\cite{spectral_theory}.
%Of course, the setting of
%uniform rectifiability and the formulation of
%Theorems~\ref{theorem:upper-main} and~\ref{theorem:funky-salmon}, do
%not carry over to such noisy distributions. However, one can still
%use multiscale geometric constructions and Menger-type curvatures to
%estimate a $d$-dimensional surface with additive noise

\subsubsection*{Discrete Curvatures of General Metric Spaces}
Immo Hahlomaa~\cite{H05} has formed a Menger-type curvature of
one-regular measures in metric spaces using the Gromov product.
He~\cite{H07} and also Raanan Schul~\cite{S07} have used it to
characterize uniform rectifiability of such measures, i.e., the
existence of a sufficiently regular curve containing their support.
The Gromov product is indeed a natural quantity for this purpose
since it is quasi-isometric to distances to ``geodesic curves'' in
any fixed ball~\cite{vaisala05}. We inquire about a similar quantity
that can be used to form $d$-dimensional Menger-type curvatures,
where $d>1$, for characterizing uniform rectifiability in some
non-Euclidean metric spaces. In this case uniform rectifiability can
be defined by a parametrizing regular surface~\cite{Semmes99} or
alternatively by big pieces of bi-Lipschitz images at all relevant
scales and locations~\cite{schul-bilip, schul-bilip-supplement}.

%\medskip
%

\subsubsection*{\bf Multi-manifold data modeling and
applications} %%
Insights of Proposition~\ref{proposition:upper-bound-big-scale} and
Theorem~\ref{theorem:doozy} are used
in~\cite{spectral_theory,spectral_applied} to solve the problem of
hybrid linear modeling. In this setting, data is sampled from a
mixture of affine subspaces (with additive noise and outliers) and
one needs to cluster the data appropriately.
This problem generalizes to the setting of multi-manifold modeling,
where affine subspaces take the form of manifolds. It also further
extends to the case where the data is embedded in metric space and
not necessarily a Euclidean space (here affine subspaces are
replaced by geodesic surfaces). We believe that any extension of the
theory presented in this paper and in~\cite{LW-part1} to those
general settings could be used to enhance the methods for solving
such problems (see e.g., \cite{ACL-kscc}).

\section*{Acknowledgement}
We thank the anonymous reviewers and the action editor for the very
careful reading of the manuscript and their constructive
suggestions. We thank Raanan Schul for bringing to our attention
some valuable references regarding $d$-dimensional uniform
rectifiability in metric spaces, where $d>1$. This work has been
supported by NSF grant \#0612608.

%\bibliographystyle{plain}
%\bibliography{Lerman-Whitehouse-Part2}

\end{document}